\documentclass[onefignum,onetabnum]{siamart171218}

\usepackage{lipsum}
\usepackage{amsfonts}
\usepackage{graphicx}
\usepackage{epstopdf}
\usepackage{algorithmic}
\ifpdf
  \DeclareGraphicsExtensions{.eps,.pdf,.png,.jpg}
\else
  \DeclareGraphicsExtensions{.eps}
\fi

\usepackage{url}
\usepackage{amsmath}
\usepackage{amsfonts}

\def\CE {\mathcal{E}}

\usepackage{verbatim}

\usepackage{color}
\usepackage[normalem]{ulem}
\newcommand\reduline{\bgroup\markoverwith{\textcolor{red}{\rule[-0.5ex]{2pt}{0.4pt}}}\ULon}

\usepackage{color}
\newcommand{\blue}[1]{\color{blue} #1}

\newcommand{\CD}{\mathcal{D}}
\newcommand{\CF}{\mathcal{F}}
\newcommand{\CL}{\mathcal{L}}
\newcommand{\CW}{\mathcal{W}}

\newcommand{\bi}{{\bf i}}
\newcommand{\bj}{{\bf j}}
\newcommand{\bk}{{\bf k}}
\newcommand{\bu}{{\mathbf u}}
\newcommand{\bv}{{\mathbf v}}
\newcommand{\bx}{{\mathbf x}}
\newcommand{\by}{{\mathbf y}}

\def\eps {\epsilon}

\newcommand{\N}{\mathbb{N}}
\newcommand{\RE}{\mathbb{R}}
\newcommand{\Z}{\mathbb{Z}}

\newsiamremark{remark}{Remark}
\newsiamremark{hypothesis}{Hypothesis}
\crefname{hypothesis}{Hypothesis}{Hypotheses}
\newsiamthm{claim}{Claim}

\headers{Wasserstein Distance between 2D-Histograms by Minimum Cost Flows}{F. Bassetti, S. Gualandi, and M. Veneroni}

\title{On the Computation of Kantorovich-Wasserstein Distances between 2D-Histograms by Uncapacitated Minimum Cost Flows
}

\author{Federico Bassetti\thanks{Politecnico di Milano, Dipartimento di Matematica
  (\email{federico.bassetti@polimi.it}).}
\and Stefano Gualandi
\thanks{Universit\`a degli Studi di Pavia, Dipartimento di Matematica ``F. Casorati"
  (\email{stefano.gualandi@unipv.it}, \email{marco.veneroni@unipv.it}).}
  \and Marco Veneroni\footnotemark[2]
}

\usepackage{amsopn}

\ifpdf
\hypersetup{
  pdftitle={On the Computation of Kantorovich-Wasserstein Distances between 2D-Histograms by Uncapacitated Minimum Cost Flows},
  pdfauthor={F. Bassetti, S. Gualandi, M. Veneroni}
}
\fi

\begin{document}

\maketitle

\begin{abstract}
In this work, we present a method to compute the Kantorovich-Wasserstein distance of order one between a pair of two-dimensional histograms. Recent works in Computer Vision and Machine Learning have shown the benefits of measuring Wasserstein distances of order one between histograms with $n$ bins, by solving a classical transportation problem on very large complete bipartite graphs with $n$ nodes and $n^2$ edges.
The main contribution of our work is to approximate the original transportation problem by an uncapacitated min cost flow problem on a reduced flow network of size $O(n)$ that exploits the geometric structure of the cost function. More precisely, when the distance among the bin centers is measured with the 1-norm or the $\infty$-norm, our approach provides an optimal solution. When the distance among bins is measured with the 2-norm:
(i) we derive a quantitative estimate on the error between optimal and approximate solution;
(ii) given the error, we construct a reduced flow network of size $O(n)$.
We numerically show the benefits of our approach by computing Wasserstein distances of order one on a set of grey scale images used as benchmark in the literature. We show how our approach scales with the size of the images with 1-norm, 2-norm and $\infty$-norm ground distances, and we compare it with other two methods which are largely used in the literature.
\end{abstract}



\section{Introduction}
The Transportation Problem, also called the Hitchcock-Koop\-mans transportation problem \cite{Flood1953}, is a historical problem for the mathematical programming community \cite{Schrijver2002}.
Indeed, the Transportation Problem is a particular case in the discrete setting of the more general Monge-Kantorovich Transportation Problem, largely studied in the functional analysis community (e.g., see \cite{Villani2008,AGS}),
since the pioneering work of the French mathematician Gaspard Monge in 1781 \cite{Monge1781}, and later
the fundamental work of L.V. Kantorovich on the theory of Linear Programming duality \cite{Vershik2013}.
In the discrete setting, the Transportation Problem played a fundamental role in the development of the (Network) Simplex algorithm and of all the related network flows problems \cite{Schrijver2002,Goldberg1989,Ahuja}.
In the continuous setting, the Transportation Problem has recently gathered a renewed interest in computer science and applied mathematics, as an efficient method to compute distances between probability measures.
We just mention here the attention gathered from the Machine Learning community
in the use of the Wasserstein distances within Generative Adversarial Networks (GAN) \cite{Arjovsky2017},
which was possible thanks to recent developments of very efficient numerical algorithms based on entropic
regularization \cite{Cuturi2013,Solomon2015,Chizat2016,Altschuler2017}.

In this paper, we focus on the Kantorovich metric, i.e. the Wasserstein (Vasershtein) distance of order 1, between two-dimensional histograms. By exploiting the geometric cost structure of the problem, we reduce the computation of Kantorovich distances to the problem of solving an uncapacitated minimum cost flow problem \cite{Orlin} on a graph with a prescribed topology. We consider three types of transportation costs among locations, induced by the $p$-norms with $p=1$, $p=2$, and $p=\infty$.
For each norm, we build a flow network with a prescribed topology and given costs: our aim is to keep the size of the flow network as small as possible
in order to accelerate the computation of Kantorovich distances, while limiting the use of computer memory.
In addition, for the 2-norm, we provide a flow network that permits to find an approximate solution, and we
derive a quantitative estimate on the error between optimal and approximate solution.

We stress the importance of having fast methods to compute exact and approximate Kantorovich distances.
As discussed next,
these distances are used as subproblems of more general tasks, and they are usually (re)computed again, and again, and again.
Having a strongly polynomial algorithm of worst-case complexity $O(N^3\log{N})$ (e.g., \cite{Goldberg1989}),
where $N$ is the number of locations of the transportation problem, it is definitely not enough.
In addition, the size of Transportation Problems that must be solved is so large, that we need new
numerical methods with strong mathematical foundations to solve such problems in practice.
The contribution of our paper tackles exactly this numerical challenge.

\paragraph{Related Works} In mathematical programming, the Hitchcock-Koopmans transportation problem is
considered a well--solved problem, since it can be solved with strongly polynomial algorithms
or with Network Simplex algorithms \cite{Ahuja,Goldberg1989}. Recent work focused on the structure
of the Transportation polytope \cite{Balinski1993,Borgwardt2017}, on variations of the problem with nonlinear
variables \cite{Tuy1996}, or on mixed production transportation problems with concave costs \cite{Holmberg1999}.
However, our paper follows a different (numerical) line of research, which aims at solving very
large instances of structured Hitchcock-Koopmans problems.

Our work is closely related to what, in the
computer vision and image retrieval community, is known as the {\it Earth Mover's Distance (EMD)} \cite{Rubner2000,Pele2009}.
In this context, the Kantorovich metric is used to compute the distance between images by precomputing so--called SIFT descriptors
(special type of histograms, see \cite{Lowe1999})
and then by computing distances between such descriptors. The advantages of using a more computationally
demanding metric instead of other simpler metrics is empirically demonstrated by the quality of the results,
and in addition is elegantly supported by the mathematical theory that lays the foundation of the Kantorovich metric.
Remarkable results were also obtained in the context of cell biology for comparing flow cytometry diagrams
\cite{Bernas2008,Orlova2016} and in radiation therapy to improve the comparison of histograms related to tumor shape \cite{Kazhdan2009}.

In statistics and probability, the Kantorovich distance is known as the Mallow's distance \cite{Bickel2001}.
It has been used for the assessment of goodness of fit between distributions \cite{Munk1998,Sommerfeld2018}
as well as  an alternative  to the usual $g$-divergences as  cost function in minimum distance point estimation problems \cite{Bassetti2006,Bassetti2006b}.
It was used to compare 2D histograms, but only considering the 1-norm as a cost structure
of the underlying Transportation Problem \cite{LingOkada2007}. Another interesting application
is the computation of the mean of a set of empirical probability measures based on
the Kantorovich metric \cite{Cuturi2014}.
In machine learning, the use of Kantorovich distances is increasingly spreading in different contexts.
Apart from the GAN networks \cite{Arjovsky2017}, it has been used in unsupervised domain adaptation \cite{Courty2017},
in semi-supervised learning \cite{Solomon2014}, and as a Loss Function for learning probability distributions
\cite{Frogner2015}.
For the theoretical foundations of the Monge-Kantorovich Transportation Problem, we refer the interested reader
to \cite{Santambrogio2015,Villani2008}, while for a survey on recent applications of
Optimal Transportation to \cite{Solomon2018}.

\paragraph{Outline}

In \cref{section:2} we review the Kantorovich distance (i.e., the Wasserstein distance of order 1) in the discrete setting and we show its connection with Linear Programming (\cref{ssec:lp}) and with uncapacitated minimum cost flow problems (\cref{sec:mincostflow} and \cref{Sec:flowvsKant}). In \cref{Sec:mainid}, we discuss how to reduce the time complexity of evaluating the Kantorovich distance by approximating the transportation problem with a flow problem on a specific reduced network, we define  the relative approximation error which is incurred upon by introducing a reduced flow network, and we provide a bound upon it. 

Our main contributions are presented in \cref{section:4}, where we exploit the reduced flow networks to efficiently compute
Kantorovich distances while using the 1-norm, the 2-norm, and the $\infty$-norm as ground distances. In addition, for the 2-norm, we
derive a quantitative estimate on the error between the optimal solution and an approximate solution on a reduced flow network (\cref{PropL2bound}).

In \cref{section:3}, we give all the details of the results
presented in \cref{sec:mincostflow}-\cref{Sec:mainid}.
Precisely, in \cref{ssec:duality}, we recall the classical Strong Duality Theorem, which we use in \cref{ssec:equiv} in order to establish a relation between the minimum of the cost flow problem and the infimum of Kantorovich's transportation problem. The equivalence of these two problems is well known (\cref{prop_equivalence}), but we choose to give a proof here, as we were not able to find one in any reference. In \cref{prop:lb} we state and prove the bound on the relative approximation error $\CE_G(b)$ introduced in \cref{eq:rae}. This is key in the proof of the error estimate \cref{eq:stima1} in \cref{PropL2bound}.

{\blue Finally, in \cref{section:5} we conclude by reporting our extensive numerical experiments.

}

\section{Background}
\label{section:2}

In this section, we review the basic notions and we fix the notation used in this paper. Moreover, in \cref{Sec:mainid} we state under which conditions Optimal Transportation and minimum cost flow problems yield the same optimal solutions, and we state a universal upper bound on the relative approximation error.

\newcommand{\nx}{{n_1}}
\newcommand{\ny}{{n_2}}
\subsection{Kantorovich distance in the discrete setting}
Let $X=\{x_1,\dots,x_\nx\}$ and $Y=\{y_1,\dots, y_\ny\} $ be two discrete spaces.
Given two probability vectors on $X$ and $Y$, denoted $\mu=(\mu(x_1),\dots,\mu(x_\nx))$
and $\nu=(\nu(y_1),\dots,\nu(x_\ny))$, and a cost $c : X \times Y \to \RE_+$,
the {\it Kantorovich-Rubinshtein  functional}  between $\mu$ and $\nu$ is defined as
\begin{equation}
\label{eq:kantorovich}
		\CW_c(\mu,\nu)= \inf_{ \pi \in \Pi(\mu,\nu)} \sum_{ (x,y) \in X\times Y} c(x,y) \pi(x,y)
\end{equation}
where $\Pi(\mu,\nu)$ is the set of all the probability measures on $X \times Y$ with marginals $\mu$ and $\nu$, i.e., of
the probability measures $\pi$ such that
\[
\sum_{y \in Y} \pi(x,y)=\mu(x) \quad  \text{and} \quad \sum_{x \in X} \pi(x,y)=\nu(y) 
\]
for every $(x,y)$ in $X \times Y$.
 Such probability measures are sometimes called transportation plans or couplings
for $\mu$ and $\nu$.
An important special case is when $X=Y$ and the cost function $c$ is a distance on $X$. In this case,
which is our main interest in this article,
$\CW_c$ is a distance on the simplex of probability vectors on $X$,  also known as \emph{Wasserstein distance} of order $1$.
It is worth mentioning that a Wasserstein distance of order $p$ can be defined, more in general, for arbitrary probability measures on a metric space $(X,d)$ by
\begin{equation}\label{wpgeneral}
 W_p(\mu,\nu):=\left(\inf_{ \pi \in \Pi(\mu,\nu)} \int_{ X\times X} d^p(x,y) \pi(dx dy)\right)^{\min(1/p,1)},
\end{equation}
where now $\Pi(\mu,\nu)$ is the set of all probability measures on the Borel sets of  $X \times X$ that have marginals $\mu$ and $\nu$, see, e.g., \cite{AGS}.
The infimum in \cref{wpgeneral} is attained and any
probability $\pi$ which realizes the minimum is called an {\it optimal transportation plan}.

\subsection{Linear Programming and Earth Mover's Distance}
\label{ssec:lp}
The Kantorovich-Rubinshtein  transportation problem in the discrete setting can be seen as a special case of
the following Linear Programming problem, where we assume now that $\mu$ and $\nu$ are
generic vectors of dimension $\nx$ and $\ny$, with positive components,
\begin{align}
\label{p1:funobj} (P) \quad \min \quad & \sum_{x \in X}\sum_{y \in Y} c(x,y) \pi(x,y) \\
\mbox{s.t.} \quad
\label{p1:supply} & \sum_{y \in Y} \pi(x,y) \leq \mu(x) & \forall x \in X \\
\label{p1:demand} & \sum_{x \in X} \pi(x,y)  \geq \nu(y) & \forall y \in Y \\
\label{p1:posvar} & \pi(x,y) \geq 0.
\end{align}
\noindent Note that the maximum flow quantity is equal to
\[
\sum_{x \in X}\sum_{y \in Y} \pi(x,y) = \min \left \{\sum_{x \in X} \mu(x), \sum_{y \in Y} \nu(y) \right \}.
\]
If $\sum_{x } \mu(x) = \sum_{y} \nu(y)$ we have the so-called  {\it balanced} transportation problem, otherwise the transportation problem is said to be {\it unbalanced}.
For balanced optimal transportation problems, constraints \cref{p1:supply} and \cref{p1:demand} must be satisfied with equality, and the problem
reduces to the Kantorovich transportation problem (up to normalization of the vectors $\mu$ and $\nu$).
%
%
Problem (P) is also related to the so-called  {\it Earth Mover's distance}.
In this case, $X,Y \subset \RE^d$ and  $x_i$ ($y_j$, respectively)  is  the center of the data cluster $i$  ($j$, respectively). Moreover,  $\mu(x_i)$ ($\nu(y_j)$, respectively)  is the number of
points in the cluster $i$ ($j$, respectively) and, finally,  $c(x_i,y_j)$  is some measure of dissimilarity between
$x_i$ and $y_j$.
 Once the optimal transportation  $\pi^*$ is determined, the Earth Mover's distance
between the {\it signatures} $(x_i,\mu(x_i))_i$  and $(y_j,\nu(y_j))_j$ is defined as
\[
EMD(\mu,\nu)= \frac{\sum_{x \in X}\sum_{y \in Y} c(x,y) \pi^*(x,y)}{\sum_{x \in X}\sum_{y \in Y}  \pi^*(x,y)}.
\]
The Earth Mover's distance (EMD) was first introduced
by Rubner et al. for color and texture images in \cite{Rubner1998} and \cite{Rubner2000}.

As noted in \cite{Rubner2000},  the EMD is a true metric on
distributions, and it is exactly the same as the Wasserstein distance, also known as Mallows distance.
For more details on this, see e.g. \cite{Bickel2001} and the references therein.
The cost function used in the EMD is typically a true distance, in general an $L_p$ distance.

\subsection{Uncapacitated minimum cost flow problem on a graph}
\label{sec:mincostflow}

As we shall see in \cref{Sec:flowvsKant},
a standard way to solve problem (P) is to recast it as an uncapacitated minimum
cost flow problem.
For this reason,  we briefly recall the definition of {\it minimum cost flow problem} on a directed graph.

Let $G=(V,E)$ be a directed network with no self-loops, where $V$ is the vertex set and $E$ the set of edges. 
Consider  a cost function  $c: E \to [0,+\infty)$  and a function  $b: V \to \RE$ such that $\sum_{v  \in V} b(v)=0$.
A flow (or $b$-flow) on $G$ is any  function $f: E \to [0,+\infty)$ such that
 \begin{equation}
 \label{bflow}
 \sum_{v:  (u,v) \in E}  f(u,v)- \sum_{v:  (v,u) \in E}  f(v,u) =b(u) \quad \forall \,\, u \in V.
 \end{equation}
  In the following, we denote by $\CF(G,b)$ the class of all the $b$-flows on $G$.
The  {\it (uncapacitated) minimum
cost flow problem} associated to $(G,c,b)$ consists in finding a $b$-flow $f$ which minimizes the cost
\[
\sum_{ (u,v) \in E} c(u,v) f(u,v).
\]
See, e.g., \cite{Ahuja} or Chapter 9 in \cite{KorteVygen}.
In what follows, we set
\begin{equation}
\label{minflowprim}
F_{G,c}(b):= \min_{f \in \CF(G,b) } \sum_{ (u,v) \in E} c(u,v) f(u,v).
\end{equation}
It is easy to see that if $G$ is a subgraph of $G'$, then
\begin{equation}
\label{simpleineq}
F_{G,c}(b) \geq F_{G',c}(b).
\end{equation}

\subsection{Wasserstein distance of order one as a minimum cost flow problem}
\label{Sec:flowvsKant}

It is easy to   re-write $\CW_c(\mu,\nu)$ as a minimum cost flow problem.
Consider the bipartite graph
\[
G_{X \to Y}=(X \cup Y,E_{X \to Y})
\]
 with
$E_{X \to Y}=\{ (x,y): x \in X, y \in Y \}$,
and define
\[
   b(u):=\left\{
   			\begin{array}{rl}
				\mu(u) &\text{if }u \in X,\\
				-\nu(u) &\text{if }u \in Y.
			\end{array}
			\right.
\]
It is plain to check that
 \begin{equation}\label{W=G}
 F_{G_{X \to Y},c}(b)=\CW_c(\mu,\nu).
\end{equation}
Incidentally, an uncapacitated minimum cost flow problem of this kind is known as {\it Hitchcock problem} \cite{Flood1953}.

In general, given a network $G=(V,E)$ with $n=|V|$ nodes and $k=|E|$ arcs, Orlin's algorithm \cite{Orlin} solves
the uncapacitated minimum  cost flow problem $F_{G,c}(b)$ for integer supplies/demands $b$ in $O(n \log (\min\{n, D\})S(n,k))$, where
$D := \max_{v \in V}  \{|b(v)|\}$  and  $S(n, k)$ is the complexity of finding a shortest path
 in $G$. Using Dijkstra's \cite{Dijkstra}  algorithm $S(n,k) = O(n  \log n + k)$.
This gives, for general supplies/demands,  the cost $O(n \log (n)(n \log n+ k))$.
Combining \cref{W=G} with the previous observations, one obtains  that $\CW_c(\mu,\nu)$ can be computed  exactly
with a  time complexity  $O((\nx+\ny)^3 \log(\nx+\ny))$.

When $X=Y$,  with $|X|=|Y|=:n$, and $c$ is a distance, one can show that the flow problem
associated with $\CW_c(\mu,\nu)$ can be formulated on an auxiliary network, which is smaller than the bipartite graph $G_{X \to X}$.
More precisely, set  $b=\mu-\nu$ and let
$K_n$ be the complete (directed) graph on $X$ (without self-loops).
If $c$ is a distance, then
\begin{equation}\label{W=F}
\CW_c(\mu,\nu)=  F_{K_n,c} (b).
\end{equation}
A detailed proof of this fact is given in  \cref{prop_equivalence} in \cref{section:3}.

Although $K_n$ is smaller than  $G_{X \to X}$,  the problem of determining
 $F_{K_n,c} (b)$ is  computationally demanding for relatively small $n$.
Using the previous considerations on the computational cost of a minimum flow problem, it is easy to see that
the time complexity of  computing $F_{K_n,c} (b)$ is  $O(n^3 \log(n))$
and hence it is of the same order of the complexity of computing  $F_{G_{X \to X},c}(b)$.

Even if  \cref{W=F} is not  satisfactory by itself, it is the starting point of our further developments.
In \cref{section:4}
 we shall exploit the cost structure (for some special costs) to reduce the number  of edges in the graph on which the flow problem is formulated, and to largely simplify the complexity of the problem when $X$ is the space of base points of a 2D histogram.

\subsection{Relaxation and Error bounds: main ideas}
\label{Sec:mainid}

Before moving to our main results concerning the computation of  the $\CW_c$
distance  between 2D-histograms, let us briefly outline our main ideas
 in a more general setting. All the details
are postponed to  \cref{section:3}.

Let  $G=(V,E)  \subset {K}_n$  be a directed graph
 and consider the corresponding flow problem as an approximation of
 the original flow problem on the complete graph ${K}_n$.
One has that
$F_{G,c} (b) \geq F_{{K}_n,c} (b)$,
which means that the relative error of approximation between the original minimum
and its relaxation $F_{G,c}$ satisfies
\begin{equation}
\label{eq:rae}
	\CE_G(b):=\frac{  F_{G,c} (b)-F_{{K}_n,c} (b)}{ F_{G,c} (b)} \in [0,1].
\end{equation}
If one is able to find a subgraph $G \subset {K}_n$ for which
$\CE_G(b) =0$,  one can reduce the  problem of computing  $F_{{K}_n,c} (b)$
to a simpler minimum flow  problem, i.e. to the evaluation of   $F_{G,c} (b)$. Even if this is not possible,
one can still use a subgraph $G$ to approximate $F_{{K}_n,c} (b)$ and
  control the relative error of approximation $\CE_G(b)$.

This last goal is obtained by deriving  a universal upper bound on $\CE_G(b)$.  To this end define
\begin{equation}
\label{def:cG}
	c_G(x,y):= \min_{\gamma \in \Gamma(x,y)}  \sum_{k=0}^{|\gamma|-1}  c(\gamma_k,\gamma_{k+1})
\end{equation}
where
$\Gamma(x,y)$ is the set of directed paths in $G$ from vertex $x$ to vertex $y$ and $|\gamma|$ is the length of the path $\gamma$, composed of edges $\gamma_k$,  and set
\[
	\Gamma_{G,c} := \max \left\{ 1- \frac{c(x,y)}{c_G(x,y)} :\ (x,y)\in X^2,\ x \not =y\right\}.
\]
In \cref{prop:lb} (\cref{section:3})  we prove that, if $c$ is a distance,
\begin{equation}\label{bound1}
 0 \leq \CE_G(b) \leq \Gamma_{G,c}
\end{equation}
for any vector $b$.
Hence, combining  \cref{W=F} and \cref{bound1} we get
\[
(1-\Gamma_{G,c})  F_{G,c} (\mu-\nu) \leq   \CW_c(\mu,\nu) \leq   F_{G,c}  (\mu-\nu)
\]
for any couple of measures $(\mu,\nu)$ and any distance $c$.
In the light of these considerations, the key-point is  to find  a small graph $G \subset K_n$ for which
$\Gamma_{G,c}$ is easy to compute and/or zero.


\section{Efficient Computation of \texorpdfstring{$\CW_c$}{Wdist} distance between 2D-Histograms\label{section:4}}
Histograms can be seen as discrete measures on a finite set of points in $\RE^d$. Here we are interested in the case in which  $d=2$. To represent a two-dimensional histogram with $N\times N$ equally spaced  bins, we take without loss of generality

\[
    X =\CL_{N}:=\{ \bi=(i_1,i_2): i_1=0,1,\dots,N-1, i_2=0,1,\dots, N-1\}.
\]

One can think of each point $(i_1,i_2)$ as the center of a bin.

In this case, $\# \CL_{N}=N^2$ and  the complete directed graph $K=K_{N^2}$ has $N^2(N^2-1)=O(N^4)$ edges.
 Note that, according to last section's notation, we are considering a set of $n=N^2$ points.
 To simplify notation, we write $K=K_{N^2}$.

Following the idea   sketched in \cref{Sec:mainid}, the goal
 is to approximate $\CW_c$ (or equivalently $F_{K,c}$), with $F_{G,c}$ for a suitable small graph $G \subset K$.

When the distance among the bins is measured with the 1-norm $L_1$ or the infinity-norm $L_\infty$,
we provide an optimal solution
using a   graph $G$  with $O(N^2)$ edges.
When the distance between bins is measured with the  $L_2$ -norm:
(i) we show that an optimal solution can be obtained using  a reduced network with around half the size of the complete bipartite graph,
(ii) given an error, we provide an approximate solution
using  a reduced network of size $O(N^2)$.

In order to describe the subgraphs $G$,
we use the following notation. Given $\bi=(i_1,i_2)\in \Z \times \Z$, $\bj=(j_1,j_2)\in \Z \times \Z,$ define
$\bi + \bj = (i_1+j_1,i_2+j_2)$.
If $G=(V,E)$ is a graph and $\bi,\bj\in V$, we denote the directed edge connecting $\bi$ to  $\bj$ by $(\bi,\bj)\in E$. Let $\N_0:=\N \setminus \{0\}$, $\Z_0:=\Z \setminus \{0\}$. For $L\in \N_0$, we define the following sets of vertices
\begin{align*}
	V_0&:=\{ (1,0),(0,1),(-1,0),(0,-1)\},\\
	V_L&:=V_0 \cup \left\{ (i_1,i_2) \in \Z_0\times \Z_0 :   |i_1|\leq L,\, |i_2| \leq L, \right. \\
		&\qquad \left. \text{and $|i_1|$ and $|i_2|$ are coprime} \right\}.
\end{align*}
The condition on the divisors of $|i_1|$ and $|i_2|$ implies that  any  point of coordinates  $\bi \in V_L$ is
  ``visible" from the origin (0,0), in the sense that there is no point with integer coordinates on the line segment between the origin and $\bi$.


\subsection{\texorpdfstring{$L_1$}{L1} ground distance}\label{sec.L1}
In this section we
 recover, as a special case of our approach,    the method  proposed in \cite{LingOkada2007} to efficiently compute
the $L_1$-EMD by
solving a flow problem on a graph with $O(N^2)$ edges.
Let us consider as ground distance the $L_1$ distance on $\CL_N$, i.e.
$c=d_1$, where
\[
d_1(\bi,\bj)=|i_1-j_1|+|i_2-j_2|.
\]
In this case, as a subgraph of $K$ we choose
\[
	G_0:=(\CL_N,E_0) \quad  \text{where}  \quad E_0:=\left\{(\bi,\bi+\bj):\bi \in \CL_N,\ \bj\in V_0,\ \bi+\bj \in \CL_N\right\},
\]
see \cref{Fig1}.
 %
\begin{figure}[ht]
	\begin{center}
	\includegraphics[width=2.75cm]{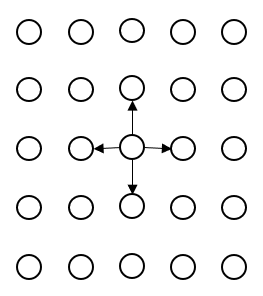}\hspace{1cm}
	\includegraphics[width=2.75cm]{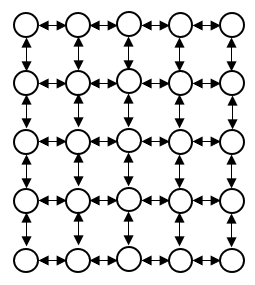}
	\caption{The edges connecting the middle node $\bi$ to the four nodes $\bi+\bj$, with $\bj\in V_0$ (left). Iterating for all $\bi \in \mathcal L_5$ we obtain $G_0$ (right). The resulting geometry is also known as {\it taxicab} or {\it Manhattan} geometry.}
	\label{Fig1}
	\end{center}
\end{figure}
    Note that the set of (directed)  edges of $G_0$ has cardinality
$2(N-1)N=O(N^2)$.

It is a simple exercise to see that in this case $\Gamma_{G_0,d_1}=0$ for every $N$. This means that one can compute
the distance $\CW_{d_1}(\mu,\nu)$ between two normalized histograms $\mu$ and $\nu$ without any error,
by solving  a  minimum cost flow problem on  the graph  $G_0$ which  has $N^2$ nodes and $O(N^2)$ edges.
We summarize the previous statements, recalling that $F_{G,c}(b)$ was defined in \cref{minflowprim}, in the next

%

\begin{proposition}
Under the previous assumptions,
 $\Gamma_{G_0,d_1}=0$ and hence
 \[
 	F_{G_0,d_1}(b)=F_{K,d_1}(b).
\]
In particular,
\[
\CW_{d_1}(\mu,\nu)=F_{G_0,d_1}(\mu-\nu)
\]
for every couple of probability measures $\mu$ and $\nu$ on $\CL_N$.
\end{proposition}

 \subsection{\texorpdfstring{$L_\infty$}{Linf} ground distance}
In this section, we specialize our approach to
deal with the $L_\infty$ ground distance.
If $c=d_\infty$, where $d_\infty(\bi,\bj)=\max(|i_1-j_1|,|i_2-j_2|)$,
we consider the graph (see \cref{Fig2})
\[
	G_1:=(\CL_N,E_1) \quad  \text{where} \quad E_1:=\left\{(\bi,\bi+\bj):\bi \in \CL_N,\ \bj\in V_1,\ \bi+\bj \in \CL_N\right\}.
\]

\begin{figure}[ht]
	\begin{center}
	\includegraphics[width=2.75cm]{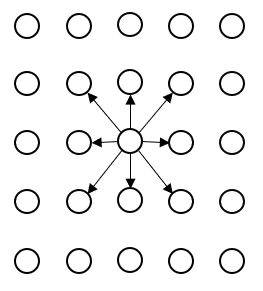}\hspace{1cm}
	\includegraphics[width=2.75cm]{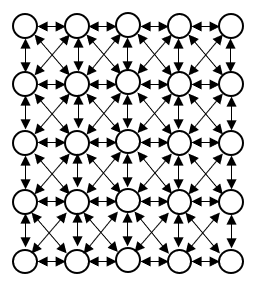}
	\caption{The edges connecting the middle node $\bi$ to the eight nodes $\bi+\bj$, with $\bj\in V_1$ (left), and the graph $G_1$ (right).
	With respect to \cref{Fig1}, four edges are added to each inner node.}
	\label{Fig2}
	\end{center}
\end{figure}

Also in this case it is easy to see that  $\Gamma_{G_1,d_\infty}=0$ and
the number of edges in $G_1$ is again $O(N^2)$, more precisely $2(N-1)(2N-1)$.

\begin{proposition}
One has $\Gamma_{G_1,d_\infty}=0$ and hence $F_{G_1,d_\infty}(b)=F_{K,d_\infty}(b)$. In particular,
\[
	\CW_{d_\infty}(\mu,\nu)=F_{G_1,d_\infty}(\mu-\nu)
\]
for every couple of probability measures $\mu$ and $\nu$ on $\CL_N$.
\end{proposition}

 \subsection{\texorpdfstring{$L_2$}{} ground distance}
 \label{ssec:L2}
 Here we face  the challenging problem of efficiently computing
 the Wasserstein distance of order 1 between 2D-histograms using  the euclidean  $L_2$ norm as ground distance.
 Let us consider $c=d_2$, where
 \[
d_2(\bi,\bj)=\sqrt{ |i_1-j_1|^2+|i_2-j_2|^2}.
\]
In order to build a suitable subgraph of $K$, we further consider the vertices sets $V_L$ defined above.
For $1 \leq L \leq N-1$, we define the graph
 \begin{equation}
 \label{def:GL}
	G_L:=(\CL_N,E_L)  \quad \text{with} \quad  E_L:=\left\{(\bi,\bi+\bj):\bi \in \CL_N,\ \bj\in V_L,\ \bi+\bj \in \CL_N\right\},
\end{equation}
 (we say that $G_L$ is the graph induced by $V_L$). See \cref{Fig3}.
%

\begin{figure}[ht]
	\begin{center}
	\includegraphics[width=2.75cm]{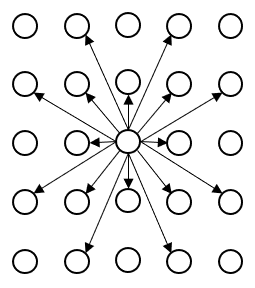}
	\caption{The set of edges  $\left\{(\bi,\bi+\bj)\right\}$, where $\bi$ is the middle node and $\bj\in V_2$. The graph $G_2$ is now too crowded to be clearly visualized.}
	\label{Fig3}
	\end{center}
\end{figure}

Our results concerning the $L_2$ ground distance are the following: in \cref{PropL2} we show that when $L=N-1$ (i.e., the induced subgraph is $G_{N-1}$)  the approximation error is zero. Again, this means that one can compute $\CW_{d_2}(\mu,\nu)$ by solving  a  minimum cost flow problem on  a graph  which is strictly contained in the complete graph $K$. Unfortunately, in this case, the number of  required edges  is of order $N^4$ (\cref{PropL2}). To conclude, in \cref{PropL2bound}, we provide a sharp estimate on the error due to the approximation of $K$ with $G_L$, for $L<N-1$.

\begin{proposition}
\label{PropL2}
With the previous notation, if  $c=d_2$, then
$c_{G_{N-1}}=d_2$, hence $\Gamma_{G_{N-1},d_2}=0$ and
\[
\CW_{d_2}(\mu,\nu)=F_{G_{N-1},d_2} (\mu-\nu)
\]
for every couple of probability measures $\mu,\nu$ on $\CL_N$.
Moreover, if $E_{N-1}$ denotes the set of edges of $G_{N-1}$, then
\begin{equation*}
	\lim_{N \to + \infty}   \frac{\# E_{N-1}}{N^4}  = \frac{6}{\pi^2} .
\end{equation*}
\end{proposition}

The next result, in combination with \cref{bound1}, shows that  one can approximate the true value
$\CW_{d_2}(\mu,\nu)$ by computing
$F_{G_{L},d_2} (\mu-\nu)$
  with $L < N-1$ and
obtain an explicit  bound on the relative error.

\begin{theorem}
\label{PropL2bound}
Let $c=d_2$,  then for every $1 \leq L < N-1$
\begin{equation}
\label{eq:stima1}
	1-\frac{\sqrt{1+4L^2}}{L+\sqrt{1+L^2}\ }   \leq \Gamma_{G_{L},d_2} \leq  c(L)\left(1-\frac{L}{\sqrt{1+L^2}}\right),
\end{equation}
where $0.25<c(L)<0.26$ and $c(L)$ is monotone decreasing. Asymptotically,
\begin{equation}
\label{eq:asympt}
		 \Gamma_{G_{L},d_2}  =  \frac{1}{8L^2} -\frac{11}{128L^4}+O\left( \frac{1}{L^6}\right) \quad \text{ for } L\to +\infty.
\end{equation}
\end{theorem}

\subsection{Proofs of \texorpdfstring{\cref{ssec:L2}}{L2}}
Define the set
\[
	\tilde{V}_L:=\left\{ (i_1,i_2)\in V_L : i_1\geq i_2\geq 0\right\}.
\]
By restricting to $\tilde{V}_L$, we consider only directions characterized by  angles $\alpha$ ranging between  $0$ and $\pi/4$; all other directions may be obtained from $\tilde{V}_L$ by rotations of $k\pi/4$, for $\ k=1,\ldots, 7$. 
To every element $\bi \in \tilde{V}_L$ we associate the slope
\[
	t_\bi : =\frac{i_2}{i_1}.
\]
The collection of these slopes, for $L\geq 1$, forms the so-called Farey sequence $\mathfrak{F}_L$, see
e.g. Section 4.5 in \cite{Knuth} or Chapter 6 in \cite{Niven}.
For example, for $L=1,2,3$, we have
\[
	\mathfrak{F}_1=\Big \{\frac{0}{1}, \frac{1}{1} \Big\}, \quad \mathfrak{F}_2=\Big \{ \frac{0}{1},\frac{1}{2}, \frac{1}{1}\Big \}, \quad
	\mathfrak{F}_3=\Big \{\frac{0}{1}, \frac{1}{3},\frac{1}{2}, \frac{2}{3},\frac{1}{1} \Big\}.
\]
Since the elements in $V_L$ are pairwise coprime, the mapping
\begin{equation*}
	t:\tilde{V}_L \to \mathfrak{F}_L,\quad\quad	 \bi \mapsto \frac{i_2}{i_1}
\end{equation*}
is one-to-one, i.e., there is a bijection between $\tilde{V}_L$ and $\mathfrak{F}_L$. We will therefore use some known properties of the Farey sequence to prove an estimate on the graph induced by $V_L$. Before proving \cref{PropL2}, we need the following remark.
 \begin{lemma}
\label{lemma:mi}
Each point $\bj$  in the square lattice $\CL_L$
can be written as
\[
	\bj= m \bi
\]
where $m$ is an integer and $\bi$ is in $V_{L-1}$.
\end{lemma}
\begin{proof}
Let $\bj=(j_1,j_2)\in \CL_L$. By definition of $\CL_L$, $0\leq j_1,j_2\leq L-1$. If $j_1$ and $j_2$ are coprime, then $\bj \in V_{L-1}$ (i.e., we can choose $m=1$, $\bi=\bj$). If they are not coprime, then there exists $m\in \N$ and $\bi \in V_{L-1}$ such that $(j_1,j_2) = (mi_1,mi_2)$.
\end{proof}

We will also use known
results on  the asymptotic density of coprime numbers.
Given a compact
convex set D in $\RE^2$ containing the origin,
the number of
primitive lattice points in the ``blow up" set $\sqrt{x} D$,
i.e. the number of coprime numbers  in  $\sqrt{x} D$,
diverges as $x$ goes to $+\infty$ as
$6 x a(D) /{\pi^2}$, where $a(D)$ is the area of $D$.
Under suitable  regularity assumptions on $D$,  one can derive precise estimates on the remainder, see e.g. \cite{Huxley96}.

\begin{lemma}\label{asymtden}
Let $\CD$ be the set of squares  in $\RE^2$ containing $(0,0)$ with side of length 1.
For any $D \in \CD$ and  $x>0$ set  $B_D(x)=\# \{ z=(z_1,z_2) \in Z_0^2 :
 z/\sqrt{x} \in D:  gcd(z_1,z_2)=1 \}$.
Then, there is a constant $C$ such that for every $x \geq1$
\[
\sup_{D \in \CD} \Big | B_D(x)-\frac{ 6 x }{\pi^2} \Big | \leq C \sqrt{x}[1+ \log(\sqrt{x})].
\]
\end{lemma}

\begin{proof}
One can write
$B_D(x)=\# \{ z \in Z_0^2 : Q(z) \leq x, gcd(z_1,z_2)=1 \} $
where
\[
Q(z)=\Big ( \inf\{t>0:  z/t \in D  \}  \Big )^2.
\]
Set also  $A_D(x)=\# \{ z \in Z_0^2 : Q(z) \leq x \} $. Combining these definitions with
 well-known properties of the M\"obius function $\mu$,
 one can prove that
\[
B_D(x) =\sum_{m \in \N }   A_D\Big ( \frac{x}{m^2} \Big )  \mu(m).
\]
See e.g. \cite{Huxley96} or Lemma 4 in \cite{Zagier}.
Since $D$ is a square of length $1$ containing the origin, then
  $ A_D\Big ( \frac{x}{m^2} \Big ) =0$ whenever $m>\sqrt{x}$.
Moreover,
\[
\Big | A_D(x)- x  \Big | \leq 4 \sqrt{x}.
\]
Combining these facts,
\[
B_D(x) = x  \sum_{m \leq \sqrt{x} }    \frac{\mu(m)}{m^2} +
R(x)
\]
where, recalling also that $|\mu(m)|\leq 1$,
\[
|R(x) | \leq  \sum_{m \leq \sqrt{x} } \left  | \Big (  A_D \Big (\frac{x}{m^2} \Big )-  \frac{x}{m^2} \Big ) \mu(m)  \right | \leq  4  \sqrt{x}   \sum_{m \leq \sqrt{x} }  \frac{1}{m}.
\]
To conclude recall that  $\sum_{m \geq 1 }    \frac{\mu(m)}{m^2}= 6/\pi^2$ to write, for $x \geq 1$,
\[
 \Big | B_D(x)-\frac{ 6 x }{\pi^2} \Big | \leq 4  \sqrt{x}   \sum_{m \leq \sqrt{x} }  \frac{1}{m}
 +x  \sum_{m > \sqrt{x} }    \frac{1}{m^2}
 \leq C [ \sqrt{x}  (1+ \log \sqrt{x}  )   ].
\]
\end{proof}

\begin{proof} (Proof of \cref{PropL2})
Let us start
by proving that
\[
	\CW_{d_2}(\mu,\nu)=F_{G_{N-1},d_2} (\mu-\nu).
\]
By definition \cref{def:GL},
\[
	G_{N-1}=(\CL_{N},E_{N-1}),\qquad E_{N-1}:=\left\{(\bi,\bi+\bj):\bi \in \CL_N,\ \bj\in V_{N-1},\ \bi+\bj \in \CL_N\right\}.
\]
Given $\bi,\bj \in \CL_N$, an admissible path $\gamma$ connecting $\bi$ to $\bj$ in $G_{N-1}$ is a collection of edges $\gamma_m=(\bx_m,\by_m)\in E_{N-1}$, $m=1,\ldots,M$ such that
$\bx_1=\bi$, $\by_m=\bx_{m+1}$, and $\by_M=\bj$. The length of a path $\gamma=\{\gamma_1,\ldots,\gamma_M\}$ is the sum of the lengths of its edges, that is
\[
	|\gamma|=\sum_{\gamma_m \in \gamma}|\gamma_m|=\sum_{m=1}^M d_2(\bx_m,\by_m).
\]
We denote by $\Gamma(\bi,\bj)$ the set of all admissible paths connecting $\bi$ to $\bj$. We want to prove that
\[
	d_2(\bi,\bj)=c_{G_{N-1}}(\bi,\bj)\quad \forall\, \bi,\bj \in \CL_{N},
\]
that is, according to definition \cref{def:cG}, that for every $\bi,\bj \in \CL_{N}$ there exists $\gamma \in \Gamma(\bi,\bj)$ such that $d_2(\bi,\bj)=|\gamma|$. Let $\bi,\bj \in \CL_N$, since
\[
	\max\{|j_1-i_1|,|j_2-i_2|\}\leq N-1,
\]
then the vertex $(|j_1-i_1|,|j_2-i_2|)$ belongs to $\CL_N$.
By Lemma \cref{lemma:mi} there exists $\mathbf v \in V_{N-1}$ such that $\bj-\bi=m\mathbf v$ (the coordinates in $V_L$ are not restricted to positive integers). Choosing as admissible path
\[
	\gamma = \{(\bi,\bi+\mathbf v),(\bi+\mathbf v,\bi+2\mathbf v),\ldots,(\bi+(m-1)\mathbf v,\bi+m\mathbf v)\},
\]
we compute
\begin{align*}
	c_{G_{N-1}}(\bi,\bj)& \leq |\gamma| = \sum_{k=0}^{m-1}d_2(\bi+k\bv,\bi+(k+1)\bv) \\
		&= md_2(\mathbf 0,\bv)
		= d_2(\mathbf 0,m\bv)
		=d_2(\bi,\bj).
\end{align*}
Since $c_{G_{N-1}}(\bi,\bj) \geq d_2(\bi,\bj)$ for all $\bi,\bj\in \CL_N$, we proved $c_{G_{N-1}}= d_2$.
%
%
%
Now we prove the second part of the statement.
Let $D_{ij,N}=\{ (y_1,y_2)  \in \RE^2:  -i \leq y_1 \leq N-1-i ,  -j \leq  y_2 \leq  N-1-j \}$.
The number of coprime integers in $D_{ij,N}$  is $B_{D_{ij,N}}(1)$,
moreover
 $B_{D_{ij,N}}(1)=B_{D_{ij,N}^*}((N-1)^2)$, where
 $D_{ij,N}^*=\{ (y_1,y_2) \in \RE^2:  -i/(N-1) \leq y_1 \leq 1-i/(N-1) ,  -j/(N-1) \leq   y_2 \leq 1-j/(N-1) \}$
 belongs to $\CD$.
Hence,  using  \cref{asymtden},
\[
B_{D_{ij,N}^*}((N-1)^2) = \frac{ 6 (N-1)^2 }{\pi^2} + \eps_{ij,N}
\]
with $ |\eps_{ij,N}| \leq C N[1+ \log(N)]$ for every $i$ and $j$. Hence,
\[
\frac{\# E_{N-1}}{N^4} = \sum_{i=0}^{N-1} \sum_{j=0}^{N-1}  B_{D_{ij,N}}(1)
=  \frac{ 6 (N-1)^2}{\pi^2} \frac{1}{N^2} +
 \sum_{i=0}^{N-1} \sum_{j=0}^{N-1}     \frac{\eps_{ij,N}}{N^4}
\]
and
\[
\Big | \sum_{i=0}^{N-1} \sum_{j=0}^{N-1}     \frac{\eps_{ij,N}}{N^4}  \Big |  \leq C \frac{1+ \log(N)}{N},
\]
so that
\[
 \lim_N \frac{\# E_{N-1}}{N^4} = \frac{ 6 }{\pi^2}.
\]
\end{proof}

In order to prove \cref{PropL2bound} we need some preliminary results.
Let $\bi$ and $\bj$ be two vectors in $\tilde V_L$ with $t_\bi < t_\bj$. We say that $\bi$ and $\bj$ are \emph{adjacent} if 
\[
	\left\{ {\bk} \in \tilde V_L : t_\bi < t_{\bk} < t_\bj \right\} =\emptyset.
\]
We say that a vertex $\bk=(k_1,k_2)\in \CL_N$  \emph{lies inside the convex cone} defined by $\bi$ and $\bj$ if
\begin{equation}
\label{def:cc}
 	t_\bi \leq \frac{k_2}{k_1} \leq t_\bj .
\end{equation}


\begin{lemma}
\label{lemma:integer_lc}
Let $L \leq N$. Given two adjacent vectors $\bi$ and $\bj$ in $\tilde V_L$ and a vector
$\bk$ in $\CL_N$ that lies inside the convex cone defined by $\bi$ and $\bj$,  there are non-negative integers $A$ and $B$ such that
\[
\bk=A\bi+B\bj.
\]
\end{lemma}

  \begin{proof}
  Let $t_{\bi}=i_2/i_1< j_2/j_1=t_{\bj}$. Since $t_{\bi}$ and $t_{\bj}$ are consecutive in the Farey set $\mathfrak F_L$, by
 \cite[Thm. 6.1]{Niven} we have
  \begin{equation}\label{relFar}
  i_1 j_2 - i_2 j_1=1.
  \end{equation}
  The equation of the line $r$ that is parallel to $\bi$ and passes through $\bk=(k_1,k_2)$ is
  \[
  	y-k_2 = \frac{i_2}{i_1}(x-k_1).
  \]
  The point $(x_0,y_0)$ where $r$ intersects $y=x j_2/j_1$ is
\[
	x_0= \frac{  i_1 k_2 -  i_2 k_1}{  i_1 j_2 - i_2 j_1} j_1,
		\quad y_0= \frac{  i_1 k_2 - i_2 k_1}{  i_1 j_2 - i_2 j_1} j_2.
\]
Let $B:=\frac{  i_1 k_2 -  i_2 k_1}{  i_1 j_2 - i_2 j_1}$. By \cref{def:cc} and \cref{relFar} $B$ is integer and nonnegative. In the same way, the intersection between the line that is parallel to $\bj$ and passes by $\bk$, and the line $y=x i_2/i_1$, is a point
\[
	(x_1,y_1)=A( i_1, i_2),\quad \text{where}\quad A=\frac{ j_2 k_1 - j_1 k_2}{  i_1 j_2 - i_2 j_1} \in \N.
\]
We conclude that $\bk = (x_1,y_1)+(x_0,y_0) = A\bi +B\bj$, with $A,B \in \N$.
    \end{proof}
\begin{lemma}
\label{lemma:maxangle}
Let $L\in \N$ be fixed. Let
\begin{equation*}
	\beta_k\text{ be the angle between } \bu_k=\left(1,\frac kL\right) \text{ and } \bu_{k+1}=\left(1,\frac{k+1}{L}\right),
\end{equation*}
for $k=0,\ldots,L-1$.
Then
\begin{equation}
\label{eq:maxangle}
		\cos(\beta_k)< \cos(\beta_{k+1})\quad \text{for }k=0,\ldots,L-1.
\end{equation}
\end{lemma}
\begin{proof}
A straightforward computation shows that
\[
	\cos(\beta_k) = \frac{\bu_k\cdot\bu_{k+1}}{|\bu_k||\bu_{k+1}|}=\frac{L^2+k(k+1)}{(L^2+k^2)^{1/2}(L^2+(k+1)^2)^{1/2}}.
\]
Since $\frac{d}{dk} \cos(\beta_k)>0$, we conclude that $\cos(\beta_k)$ is monotone increasing.
\end{proof}

\begin{proof} (Proof of \cref{PropL2bound}) In order to prove the upper bound in \cref{eq:stima1}, we first draw an estimate on a continuous approximation of  $\Gamma_{G_{L},d_2}$. Let
\[
	\bi=(i_1,i_2),\quad \bj=(j_1,j_2),\quad \bk=x\bi+y\bj,\quad \text{with}\quad i_1,i_2,j_1,j_2,x,y\in \mathbb R_+
\]
and
\begin{equation*}
	|\bi|=\sqrt{i_1^2+i_2^2} = 1,\quad |\bj|=\sqrt{j_1^2+j_2^2} =1.
\end{equation*}
A straightforward computation shows
\[
	\delta(x,y,\theta):= \frac{|\bk|}{|x\bi|+|y\bj|}=\frac{\sqrt{x^2+y^2 +2xy \theta}}{x+y},
\]
where $\theta:=\bi\cdot\bj=i_1j_1+i_2j_2$. For fixed $x,y$,
\begin{equation*}
\text{the mapping $\theta \mapsto \delta(x,y,\theta)$ is a monotone increasing function}
\end{equation*}
 and, since $0\leq \bi\cdot\bj \leq 1$, one immediately gets the bound
\[
	\frac{1}{\sqrt 2}\leq \inf_{x,y\in \mathbb R_+ } \frac{\sqrt{x^2+y^2}}{x+y} =\inf_{x,y\in \mathbb R_+ } \delta(x,y,0)\leq \delta(x,y,1)=1.
\]
On the other hand, for fixed $\theta\in(0,1)$ we consider
\[
	f(x,y)=\delta^2(x,y,\theta)= 1+ \frac{2xy}{(x+y)^2}(\theta-1).
\]
Since
\[
	\nabla f(x,y)=2\frac{y-x}{(x+y)^3}\left( y,-x\right),
\]
the line $y=x$ is a set of stationary points for $f$. Setting $\bv=(-1,1)/\sqrt 2$, $P=(x,x)$,
\[
	f(P+t\bv)=1+\frac{1-\theta}{2}\left(\frac{t^2}{2x^2}-1\right),
\]
we see that $y=x$ is the set of global minimum points for $f$ and thus for the mapping $(x,y)\mapsto \delta(x,y,\theta)$, with
\begin{equation}
\label{eq:estimatexy}
	\min_{x,y\in \mathbb R_+ }\delta(x,y,\theta)=\sqrt{\frac{1+\theta}{2}}.
\end{equation}
For $1 \leq L \leq N-1$, denote $\tilde V_L=\{\bv_0,\bv_1,\ldots,\bv_M\}$, indexed in increasing slope order (i.e., so that  $t_{\bv_0}<t_{\bv_1}  < \ldots < t_{\bv_M}$), and let $\alpha_k$ be the angle between $\bv_k$ and $\bv_{k+1}$. By \cref{lemma:integer_lc} for all $\bk \in \CL_N$ there exist two adjacent vectors $\bv_k,\bv_{k+1}$ and two integers $A,B\in \N$ such that $\bk=A\bv_k +B \bv_{k+1}$, and therefore, by \cref{eq:estimatexy}
\begin{equation}
\label{eq:estimateak}
		 \frac{|\bk|}{|A\bv_k|+|B\bv_{k+1}|}\geq \sqrt{\frac{1+\cos(\alpha_k)}{2}}.
\end{equation}
We estimate
\begin{align}
	\Lambda_L &:= \min\left\{  \frac{d_2(\bi,\bj)}{c_{G_L}(\bi,\bj)} :  \bi,\bj\in \CL_N,\ \bi \not =\bj\right\} \nonumber\\
		&= \min\left\{  \frac{d_2({\mathbf 0},\bk)}{c_{G_L}({\mathbf 0},\bk)} :  \bk \in \CL_N,\ \bk \not ={\mathbf 0} \right\} \nonumber\\
		&= \min_{\bk \in \CL_N, \bk \not ={\mathbf 0}} \max\left\{ \frac{|\bk|}{|A\bi|+|B\bj|} :
			A\bi+B\bj=\bk,\ \bi,\bj \in V_L,\ A,B \in \mathbb N  \right\} \nonumber\\
		& \! \stackrel{\cref{eq:estimateak}}{\geq} \min_{k=0,\ldots,M} \sqrt{\frac{1+\cos(\alpha_k)}{2}}.\label{eq:cos}
\end{align}
Let
\[
	\mathbf U_L:=\left\{\bu_j = \left(1, \frac jL\right), \, j=0,\ldots,L-1\right\},
		\ \mathbf V_L:=\left\{\bv_k = \left(1, t_{\bv_k}\right), \, \bv_k\in \tilde V_L\right\}.
\]
Clearly $\mathbf U_L \subseteq \mathbf V_L$, since the slopes $t_{\bv_k}$ include all elements in the Farey sequences $\mathfrak F_1,\ldots,\mathfrak F_L$. Define the sets of angles
\[
	\mathcal A:=\big\{\alpha_k = {\widehat{\bv_k\bv}}_{k+1} : \bv_k \in \mathbf V_L\big\},\qquad
	\mathcal B:= \big\{\beta_j = {\widehat{\bu_j\bu}}_{j+1} : \bu_j \in \mathbf U_L\big\}.
\]
Since the partitioning of $\pi/4$ induced by $\mathbf V_L$ is finer than that of $\mathbf U_L$
\[
	\max\left\{ \alpha_k : \alpha_k \in \mathbf V_L\right\} \leq \max\left\{ \beta_j : \beta_j \in \mathbf U_L\right\}.
\]
Since the cosine function is monotone decreasing in $[0,\pi/4]$, from \cref{eq:cos} we obtain
\begin{align*}
	\Lambda_L&\geq \min_{k=0,\ldots,M} \sqrt{\frac{1+\cos(\alpha_k)}{2}} \geq  \min_{j=0,\ldots,L-1}\sqrt{\frac{1+\cos(\beta_j)}{2}} \\
		&\stackrel{\cref{eq:maxangle}}{=} \sqrt{\frac{1+\cos(\beta_0)}{2}} = \sqrt{\frac12+ \frac{L}{2\sqrt{1+L^2}}},
\end{align*}
which thus implies
\[
	\Gamma_{G_{L},d_2}=1-\Lambda_L \leq 1 - \sqrt{\frac12+ \frac{L}{2\sqrt{1+L^2}}}.
\]
The upper bound in \cref{eq:stima1} then follows by a simple algebraic manipulation, setting
\[
	c(L):=\left(2 + \sqrt{2+ \frac{2L}{\sqrt{1+L^2}}}\right)^{-1}.
\]

\smallskip

In order to find the lower bound in \cref{eq:stima1},  by \cref{prop:lb}, it is enough to evaluate the error $\CE_G(b)$ for a specific choice of $b=\mu-\nu$. We choose as $\mu$ a unit mass concentrated in $\bi =(L,0)$, as $\nu$ a unit mass concentrated in $\bj=(L,1)$, under the assumption that $N\geq 2L$. A simple computation yields
\[
	\bk:=\bi+\bj=(2L,1),\quad |\bk|=\sqrt{4L^2+1},\quad \frac{|\bk|}{|\bi|+|\bj|}= \frac{\sqrt{4L^2+1}}{L+\sqrt{L^2+1}},
\]
and therefore
\[
	\CE_G(b) = 1 - \frac{\sqrt{4L^2+1}}{L+\sqrt{L^2+1}}.
\]
Finally, in order to estimate the asymptotic behaviour in \cref{eq:asympt}, we compute the power expansions
\begin{align*}
	1 - \frac{\sqrt{4L^2+1}}{L+\sqrt{L^2+1}} & = \frac{1}{8L^2} -\frac{11}{128L^4}+\frac{61}{1024L^6}+O\left( \frac{1}{L^8}\right) \quad \text{ for } L\to +\infty,\\
	c(L)\left(1-\frac{L}{\sqrt{1+L^2}}\right) & = \frac{1}{8L^2} -\frac{11}{128L^4}+\frac{69}{1024L^6}+O\left( \frac{1}{L^8}\right) \quad \text{ for } L\to +\infty.
\end{align*}
Since the expansions are identical up to the $1/L^4$ term, owing to \cref{eq:stima1}, we conclude \cref{eq:asympt}.

\end{proof}


\section{Minimum flow on reduced graphs: general theory}
\label{section:3}

In this section we give all the details of the results
presented in  \cref{sec:mincostflow}-\cref{Sec:mainid}.

\subsection{Dual formulation of a minimum flow problem}
\label{ssec:duality}

%
%
In what follows we need some basic results on the dual formulation
of the minimum flow problem   \cref{minflowprim}.
Letting
\[
 \CD(G,b):=\{ \varphi: V \to \RE:  \varphi(u) - \varphi(v) \leq c(u,v) \, \, \forall \,\, (u, v) \in E    \},
\]
the  dual problem corresponding to  \cref{minflowprim} is
\[
 \max_{\varphi \in  \CD(G,b) }
 \sum_{ u \in V} b(u) \varphi(u).
\]
The {\it  Strong Duality theorem} (see e.g. Thm. 9.6 in \cite{Ahuja})  states that
\[
\max_{\varphi \in \CD(G,b)} \sum_{ u \in V} b(u) \varphi(u) =  \min_{f \in \CF(G,b) } \sum_{ (u,v) \in E} c(u,v) f(u,v).
\]
Moreover, the so-called {\it Complementary Slackness Optimality Conditions} ensure that if a $b$-flow is optimal for $(G,c)$, then
there exists a potential $\varphi$ in $\CD(G,b)$ such that $ \varphi(u) - \varphi(v) = c(u,v)$ whenever $f(u,v)>0$. Conversely, if $f$ is a $b$-flow
and there is a potential  $\varphi$ in $\CD(G,b)$ such that $ \varphi(u) - \varphi(v) = c(u,v)$ whenever $f(u,v)>0$, and that $f(u,v)=0$ if
 $ \varphi(u) - \varphi(v) < c(u,v)$, then $f$ and $\varphi$ are  optimal solutions for the primal/dual problem.
 See e.g. Theorems  9.4 and  9.8 in  \cite{Ahuja}.

\subsection{Flows on Reduced Graphs}
\label{ssec:equiv}

Assuming that $X=Y$, in place of considering the  flow problem on a graph with nodes $X \cup Y$, we start by
considering a flow problem for  $b=\mu-\nu$
on the  ``reduced" set of nodes $V=X$.
Assuming that  $\mu(x)>0$ and $\nu(x)>0$ for all $x \in X$,  the set of nodes in this problem
has cardinality $n$ while in $G_{X \to Y}$ has cardinality $2n$, where $n=|X|$. For the reduced graph with nodes set $X$, we choose, as a set of edges,
the set of all possible directed links on $X$, that is $G=K_n$,
$K_n$ being the complete (directed) graph on $X$ (without self-loops). Although this choice
still considers a large number of  edges (more precisely $n(n-1)$) it will be useful in the sequel.

The next proposition summarizes  some simple relations between the above defined  problems.

\begin{proposition}
\label{prop_equivalence}
Let $\mu$ and $\nu$ be two probability vectors on $X$ and $b:=\mu-\nu$.
\begin{itemize}
\item[(a)]
If   $c(x,x)=0$ for every $x \in X$, then
\[
 	\CW_c(\mu,\nu) \geq   F_{K_n,c} (b).
\]
\item[(b)]
If, in addition,  $c(x,y) \leq c(x,z)+c(z,y)$ for every $x,y,z$ in $X$, then
\[
	\CW_c(\mu,\nu)=  F_{K_n,c} (b).
\]
\end{itemize}
\end{proposition}

In order to prove (a) and (b), we introduce an auxiliary bipartite graph with set of nodes
$V_b:= S_b \cup D_b \subset X$,
where
$S_b:=\{ x : b(x)>0\}$ and $D_b:=\{ x: b(x)<0\}$,
and set of edges  $E_b:=\{ (x,y) \in S_b \times  D_b\}$. Denoting by $G_b$ the bipartite graph $(V_b,E_b)$,
the corresponding minimum cost flow problem can be rewritten as
\[
	F_{G_b,c}(b):= \min_{f \in \CF(G_b,b)  } \sum_{ x \in S_b, y \in D_b} c(x,y) f(x,y).
\]
The cardinality of $E_b$ is  $a^+ a^- n^2$, where $a^{\pm}$ is the fraction
of vertexes  $x$ with $b(x)>0$, $b(x)<0$  respectively, hence $a^+ a^- <1$,   while the set of edges $E_{X \to X}$ has cardinality $n^2$.

The next Lemma specifies the relations between $F_{G_b,c}(b)$, the minimum of the cost flow problem $F_{K_n,c} (b)$, and the inf of Kantorovich's transportation problem $\CW_c(\mu,\nu)$.

\begin{lemma}
\label{lemma_aux}
Let $b:X\to\mathbb{R}$ be any function.
\begin{itemize}
\item[(i)]
If   $c(x,x)=0$ for every $x \in X$, then $F_{G_b,c}(b) \geq F_{K_n,c} (b)$ and,
for $b=\mu-\nu$,
\[
	F_{G_b,c}(b) \geq \CW_c(\mu,\nu).
\]
\item[(ii)]
If, in addition,  $c(x,y) \leq c(x,z)+c(z,y)$ for every $x,y,z$ in $X$, then
\[
	F_{G_b,c}(b) =F_{K_n,c} (b).
\]
\end{itemize}
\end{lemma}
\begin{proof}
(i) Since $G_b \subset K_n$,  \cref{simpleineq} implies that $F_{G_b,c}(b) \geq F_{K_n,c} (b)$. If $f \in  \CF(G_b,c)$ with $b=\mu-\nu$,
set
\[
	\pi(x,x)= \min(\mu(x),\nu(x)),
\]
and, if $x \in S_b$ and $y\in D_b$,
\[
	\pi(x,y)= f(x,y).
\]
Finally, define $\pi(x,y)=0$ in all the remaining cases.
It is easy to check that such a $\pi$ belongs to $\Pi(\mu,\nu)$  and again the inequality
follows.

(ii) Let $\bar f$  be an optimal $b$-flow for $(K_n,c)$.
For $x$ in $S_b$ and $y$ in $D_b$  let $\Gamma(x,y)$ be  the set of directed paths in $K_n$ from $x$ to $y$ with no loops; for any $\gamma$ in $\Gamma(x,y)$
of length $|\gamma|=t$
let $\bar f(\gamma)=\min\{\bar f(\gamma_1),\dots,\bar f(\gamma_{t})\}$ be the flow through $\gamma$ from $x$ to $y$.
The total flow from $x$ to $y$ can then be defined by
\[
f^*(x,y):=\sum_{\gamma \in \Gamma(x,y)} f(\gamma).
\]
Doing this for any $(x,y)$ in $E_b$, it is easy to see that the resulting $f^*$ is a $b$-flow in $G_b$.
By the Complementary Slackness Conditions,  since $\bar f$  is an optimal $b$-flow for $(K_n,c)$, then there is  $\varphi$ in $\CD_{K_n,b}$ such that $ \varphi(u) - \varphi(v) = c(u,v)$ whenever $\bar f(u,v)>0$.
Now let $x$ in $S_b$ and $y$  in $D_b$ and $\gamma$ in  $\Gamma(x,y)$  with $\bar f(\gamma)>0$.
Since it must be that $\bar f(\gamma_i,\gamma_{i+1})>0$ for $i=1,\dots,|\gamma|-1$, it  follows that $ \varphi(\gamma_i) - \varphi(\gamma_{i+1}) = c(\gamma_i,\gamma_{i+1})$. Hence,
\[
c(x,y) \geq  \varphi(x) - \varphi(y)=\sum_{i=1}^{|\gamma|-1}   \varphi(\gamma_i) - \varphi(\gamma_{i+1}) = \sum_{i=1}^{|\gamma|-1}    c(\gamma_i,\gamma_{i+1})
\geq  c(x,y),
\]
where the last inequality follows by triangle inequality. Summarizing, we have proved that  $c(x,y)= \varphi(x) - \varphi(y)$ if  $f^*(x,y)>0$.
 Conversely, if  $c(x,y)> \varphi(x) - \varphi(y)$, then necessarily $f^*(x,y)=0$. Since clearly $\varphi$ belongs to $\CD(G_b,c)$,  then $f^*$ and $\varphi$ are optimal primal/dual solutions for $(G_b,c)$
by the Slackness Conditions.  Using the Strong Duality theorem we conclude
that $F_{K_n,c}(b)=\sum_{v \in X} \varphi(v) b(v)= F_{G_b,c}(b)$.

\end{proof}

\begin{proof} (Proof of  \cref{prop_equivalence}.)
(a) In order to prove that $\CW_c(\mu,\nu) \geq   F_{K_n,c} (\mu-\nu)$, note that if $\pi \in \Pi(\mu,\nu)$
then $f(x,y):=\pi(x,y)$ for $x\not = y$ belongs to $ \CF(K_n,b)$ for $b=\mu-\nu$. Since $c(x,x)=0$ the inequality follows.
%

(b) Let $b=\mu-\nu$. By \cref{lemma_aux}-(i) and part (a), we have that
\begin{equation}
\label{ineq_chain}
	F_{G_b,c}(\mu-\nu) \geq \CW_c(\mu,\nu) \geq F_{K_n,c}(\mu-\nu).
\end{equation}
On the other hand, by \cref{lemma_aux}-(ii),
\begin{equation}
\label{eq_chain}
	F_{G_b,c}(\mu-\nu) = F_{K_n,c}(\mu-\nu).
\end{equation}
Combining \cref{ineq_chain} and \cref{eq_chain} immediately yields
\[
	\CW_c(\mu,\nu)=F_{K_n,c}(\mu-\nu).
\]
\end{proof}

The previous result shows that one can
compute the EMD  between two normalized measures $\mu$ and $\nu$ (with respect to some ground distance $c=d$) by
solving  the minimum cost flow problem $F_{G_b,d} (\mu-\nu)$.

Although $G_b$ is smaller than  $G_{X \to X}$,  the problem of determining
 $F_{G_b,d} (b)$ is  computationally demanding for relatively small $n$.
Using the previous considerations on the computational cost of a minimum flow problem, it is easy to see that
the time complexity of  computing $F_{G_b,d} (b)$ is  $O(n^3 \log(n))$
and hence it is of the same order of the complexity of computing  $F_{G_{X \to Y},d}(b)$.

\subsection{Error Bounds}
\label{ssec:bound}
%
%
%
%
%
%
In this Section
 we prove the universal upper bound \cref{bound1}, i.e.
a bound
 on $\CE_G(b)$  which depends  only on the geometry of $(G,c)$
and not on the specific $b$.

First of all, note that if $c$ is a distance,  then $c_G$ (defined in \cref{def:cG}) is a distance, and, since $G \subseteq K_n$,
\begin{equation}
\label{eq:increas}
	c_G(x,y) \geq c(x,y)\quad \forall\, x,y\in V.
\end{equation}
Denote by $\gamma^{(x,y)}$ any optimal path in $G$ connecting $x$ to $y$, that is, any path such that
\[
	 \sum_{i=0}^{|\gamma^{(x,y)}|-1}  c\left(\gamma_i^{(x,y)},\gamma_{i+1}^{(x,y) }\right)
	 	\leq \sum_{i=0}^{|\gamma|-1}  c(\gamma_i,\gamma_{i+1})
\]
for any path $\gamma$ in $G$  from $x$ to $y$.  Then
\[
	c_G(x,y)= \sum_{i=0}^{|\gamma^{(x,y)}|-1}  c\left(\gamma_i^{(x,y)},\gamma_{i+1}^{(x,y)} \right).
\]
Note that clearly  $\gamma^{(x,y)}$ need not be unique.
The constant
\[
	\Gamma_{G,c} := \max \left\{ 1- \frac{c(x,y)}{c_G(x,y)} :\ (x,y)\in X^2,\ x \not =y\right\}
\]
provides the bound we are looking for.

\begin{proposition}
\label{prop:lb}
If $c$ is a distance,
then for every $b: X \to \RE$
\[
 0 \leq \CE_G(b) \leq \Gamma_{G,c}.
\]
\end{proposition}

\begin{proof} 
Clearly, it suffices to prove  that  $F_{G,c}(b) -F_{K_n,c} (b) \leq  \Gamma_{G,c}  F_{G,c} (b)$, i.e.,
\[
 	(1-\Gamma_{G,c})F_{G,c}(b)  \leq F_{K_n,c} (b).
\]
First of all, we notice that by definition of $\Gamma_{G,c}$
\[
	c_G(x,y) -c(x,y)  =\left(1- \frac{c(x,y)}{c_G(x,y)}\right)c_G(x,y) \leq \Gamma_{G,c}c_G(x,y),
\]
and therefore
\begin{equation}
\label{eq:stimac1}
 	(1-\Gamma_{G,c})c_G(x,y)\leq c(x,y).
\end{equation}
Let $E$ be the set of edges of $G$, for every $(x,y) \not \in E$, let $\gamma^{(x,y)}$ be  an optimal  path (with no loops) in $G$. Now let $f \in \CF(K_n,b)$, then
\begin{align*}
	\sum_{x\not =y}  f(x,y) c(x,y)  &= \sum_{(\alpha,\beta) \in E} f(\alpha,\beta) c(\alpha,\beta) +\sum_{(x,y) \not \in E} f(x,y) c(x,y)  \\
		& \geq  \sum_{ (\alpha,\beta) \in E} f(\alpha,\beta) c(\alpha,\beta) \\
		&\qquad +\sum_{(x,y) \not \in E}   f(x,y) (1-\Gamma_{G,c})
			 \sum_{i=0}^{|\gamma^{(x,y)}|-1}  c\left(\gamma_i^{(x,y)},\gamma_{i+1}^{(x,y) }\right)     \\
		&=\sum_{(\alpha,\beta) \in E} \Big [  f(\alpha,\beta) + (1-\Gamma_{G,c})
			 \sum_{(x,y) \not \in E: (\alpha,\beta) \in \gamma^{(x,y)}}  f(x,y) \Big ]   c(\alpha,\beta) \\
	  	&\geq (1-\Gamma_{G,c})   \sum_{(\alpha,\beta) \in E}   \Big[  f(\alpha,\beta) +
			 \sum_{(x,y) \not \in E: (\alpha,\beta) \in \gamma^{(x,y)}}  f(x,y)  \Big] c(\alpha,\beta),
\end{align*}
where we used also \cref{eq:stimac1} and that, by \cref{eq:increas}, $ \Gamma_{G,c} \geq 0$. Letting
\[
f'(\alpha,\beta)=
\begin{cases}
  0 & \text{if $(\alpha,\beta) \not \in E$}    \\
  f(\alpha,\beta) + \sum_{(x,y) \not \in E: (\alpha,\beta) \in \gamma^{(x,y)}}  f(x,y) & \text{if $(\alpha,\beta) \in E$},   \\
\end{cases}
\]
one can rewrite the previous inequality as
\[
\sum_{x \not =y} f(x,y) c(x,y) \geq  (1-\Gamma_{G,c})   \sum_{(x,y)\in E } f'(x,y) c(x,y).
\]
To conclude, one checks that $f'$ belongs to $\CF(G,b)$ ,
and hence
\[
\sum_{x \not =y} f(x,y) c(x,y) \geq  (1-\Gamma_{G,c}) F_{G,c} (b).
\]
\end{proof}


\section{Numerical experiments}\label{section:5}
In this section, we report the results of our numerical experiments.
The goal of our experiments is to address the following research questions:

\begin{enumerate}
\item Which is the most efficient Minimum Cost Flow algorithm from those available in the literature, considering that our instances of optimal transport have a specific geometric cost function?
\item How fast can we exactly compute the Wasserstein distance of order 1 as a function of the ground distance and as a function of the 2D histogram size?
\item How tight is in practice the bound given in \cref{PropL2bound} as a function of the parameter $L$?
\item How does our approach compare with other state-of-the-art algorithms?
\end{enumerate}

\noindent In order to answer these questions we have run several experiments using the Discrete Optimal Transport Benchmark (DOTmark) \cite{Dotmark}, which has 10 classes of grey scale images related to randomly generated images, classical images, and real data from microscopy images of mitochondria.
In each class there are 10 different grey scale images.
Every image is given at the following pixel resolutions: $32\times32$, $64\times64$, $128\times128$, $256\times256$, and $512\times512$.
\cref{dot:2} show the {\it Classical} and the {\it Microscopy} images, respectively.

\begin{table}[ht!]
\centering
\setlength\tabcolsep{1pt}
\begin{tabular}{cccccccccc}
  \includegraphics[width=1.2cm]{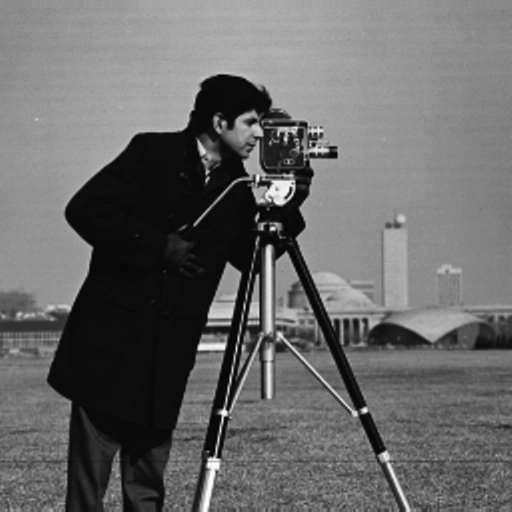}  & \includegraphics[width=1.2cm]{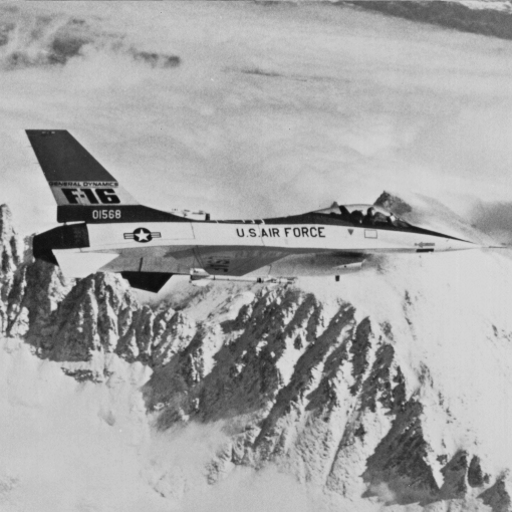} &
  \includegraphics[width=1.2cm]{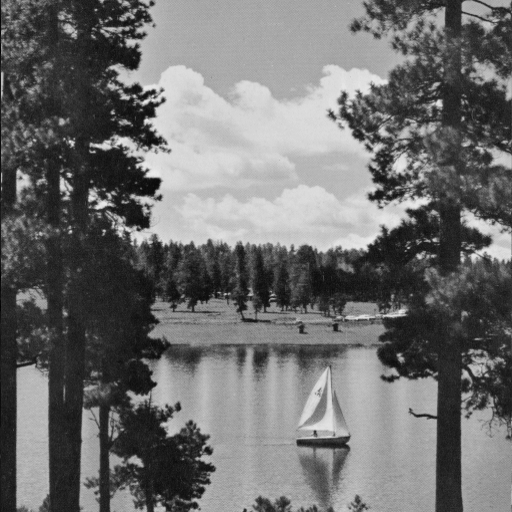} & \includegraphics[width=1.2cm]{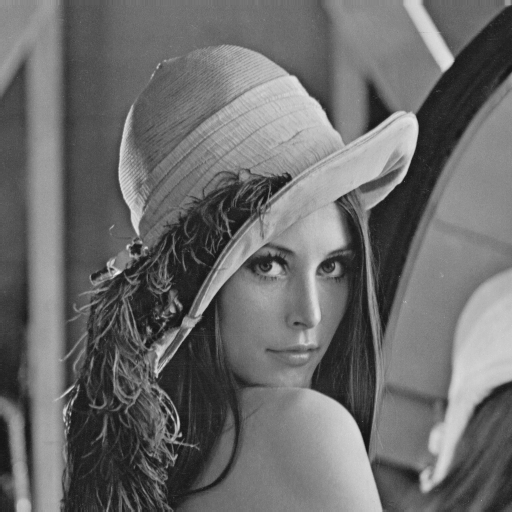} &
  \includegraphics[width=1.2cm]{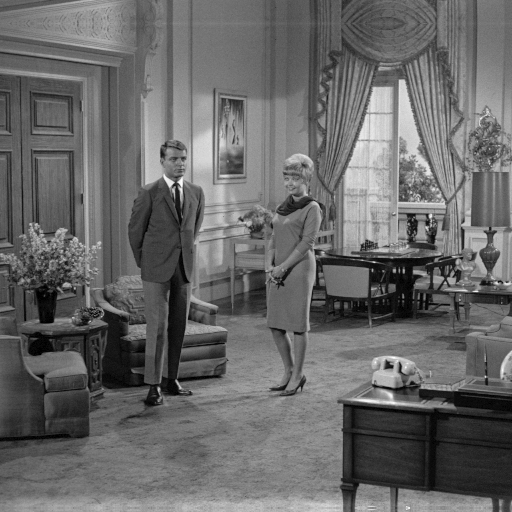} &
  \includegraphics[width=1.2cm]{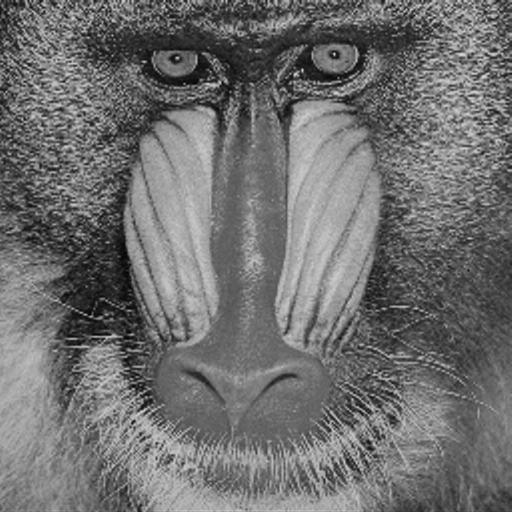}  & \includegraphics[width=1.2cm]{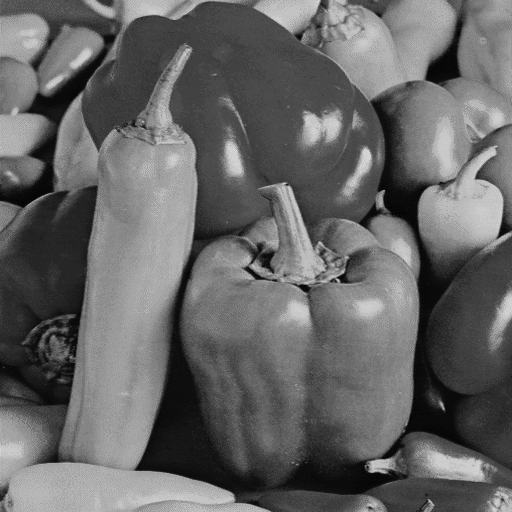} &
  \includegraphics[width=1.2cm]{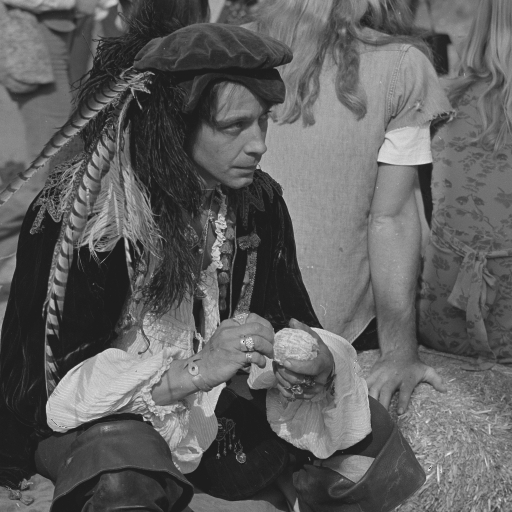} & \includegraphics[width=1.2cm]{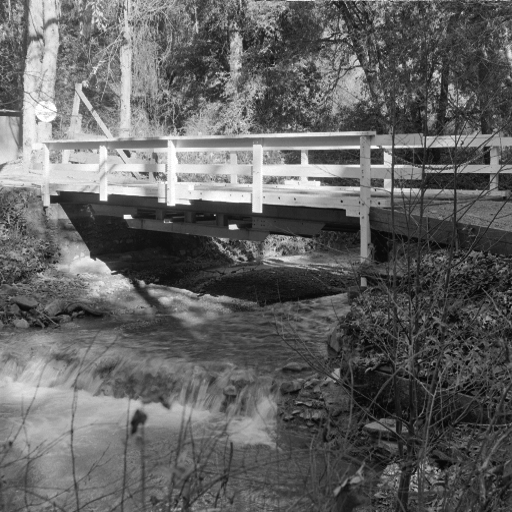} &
  \includegraphics[width=1.2cm]{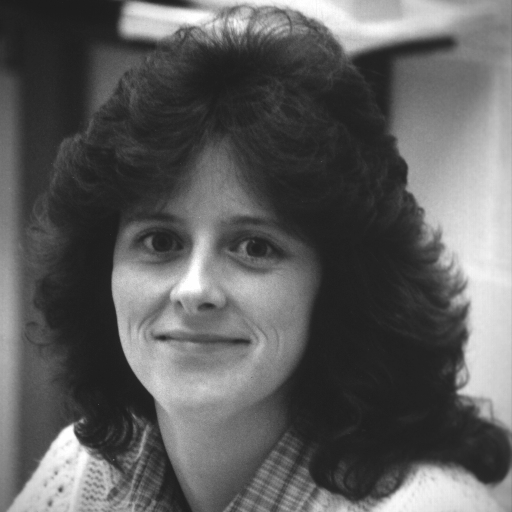} \\
  \includegraphics[width=1.2cm]{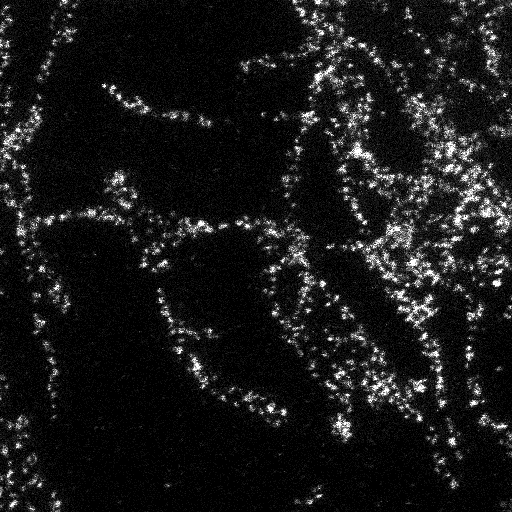}  & \includegraphics[width=1.2cm]{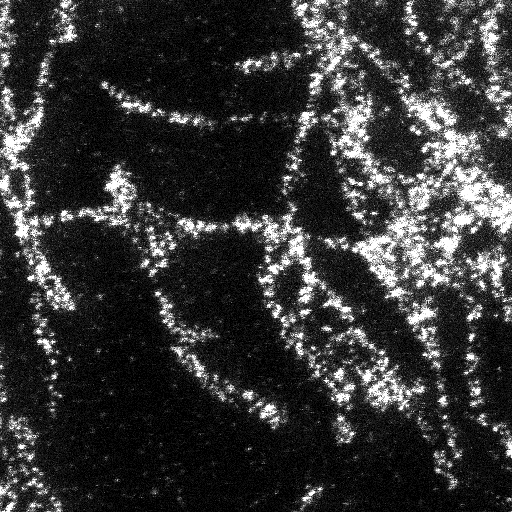} &
  \includegraphics[width=1.2cm]{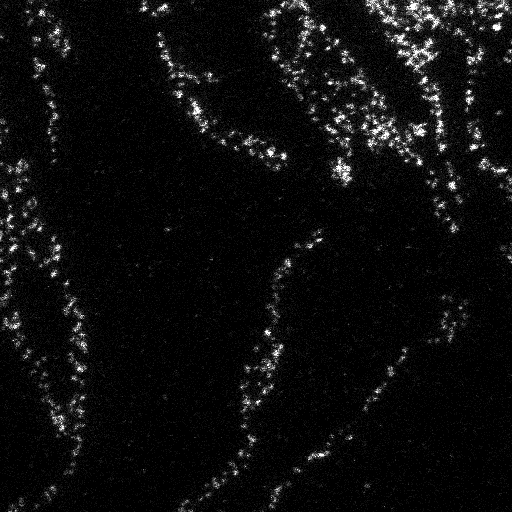} & \includegraphics[width=1.2cm]{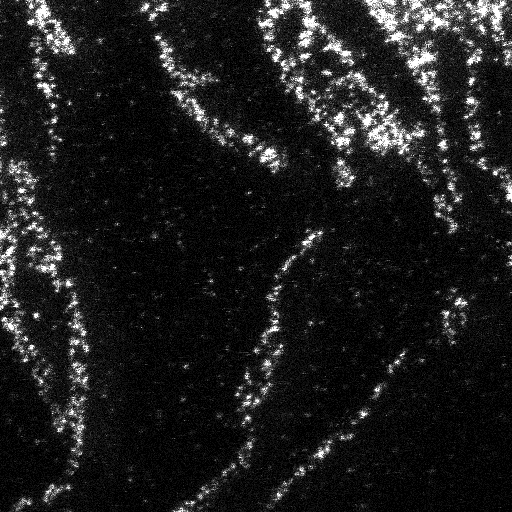} &
  \includegraphics[width=1.2cm]{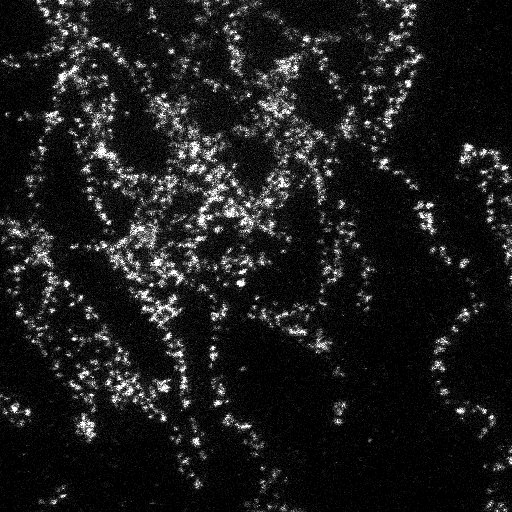} &
  \includegraphics[width=1.2cm]{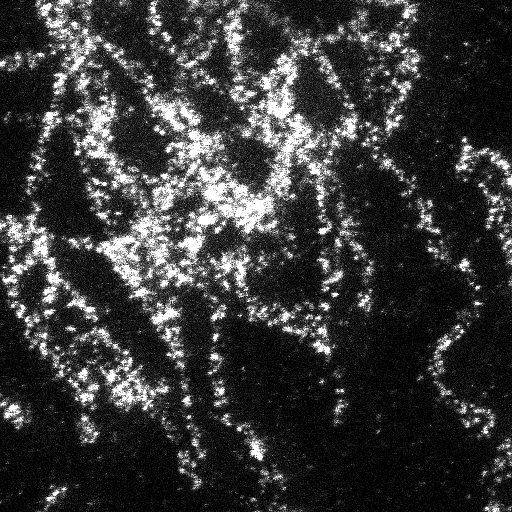}  & \includegraphics[width=1.2cm]{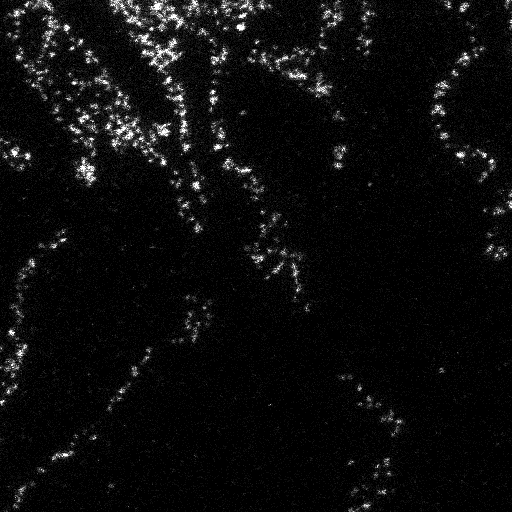} &
  \includegraphics[width=1.2cm]{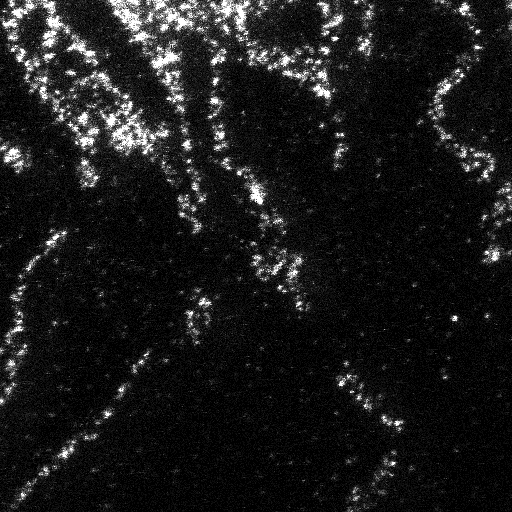} & \includegraphics[width=1.2cm]{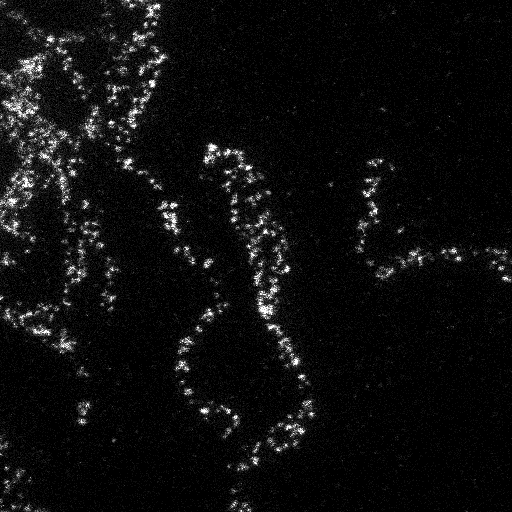} &
  \includegraphics[width=1.2cm]{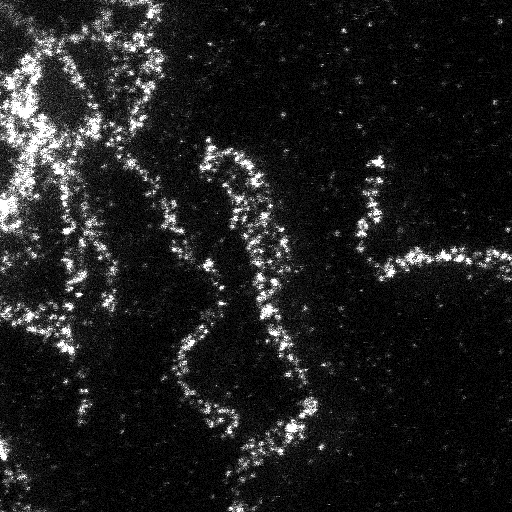} \\
  \end{tabular}
\caption{Classical images and Microscopy images of mitochondria. \label{dot:2}}
\end{table}

\paragraph{Implementation details}
We run our experiments using the Network Simplex as implemented in the Lemon C++ graph library\footnote{\url{http://lemon.cs.elte.hu} (last visited on May, 2nd, 2019)}.
The tests were executed on a Dell workstation equipped with an Intel Xeon W-2155 CPU working at 3.3 GHz and with 32 GB of RAM.
All the code was compiled with the Microsoft Visual Studio 2017 compiler.
Our C++ code is freely available at \url{https://github.com/stegua/dotlib}.

\subsection{Comparison of Minimum Cost Flow Algorithms}
As a first step, we run a set of experiments to select the fastest min cost flow algorithm for our geometric optimal transportation instances. We considered the following codes:

\begin{enumerate}
    \item {\bf EMD}: is the exact solver described in \cite{Rubner1998}, based on an implementation (in ANSI-C) of the transportation simplex method, as described in \cite{Hillier1995}.
    \item We test three different linear programming algorithms implemented in the commercial solver Gurobi v8.1: {\bf Primal Simplex}, {\bf Dual Simplex}, and {\bf Barrier}. The barrier algorithm is the only parallel algorithms able to exploit the 10 physical cores of our CPU.
    \item We test three implementations of Minimum Cost Flow algorithms implemented in the Lemon Graph Library, corresponding to three different combinatorial algorithms: {\bf Cycle Cancelling} \cite{Goldberg1989}, {\bf Cost Scaling} \cite{Goldberg1997,Bunnagel1998}, and the {\bf Network Simplex} \cite{Dantzig2016}. For a deeper and more general comparison with other implementations among combinatorial algorithms, we refer to \cite{Kovacs2015}.
\end{enumerate}

Figure \ref{Tab1} shows the results of our numerical tests. For each method, we compute the distance between any pair of images belonging the the {\it Classic} class (see Figure \ref{dot:2}), with size $32\times32$, for a total of 45 problem instances. As ground distance, we use the Euclidean distance, and we solve the corresponding problem defined on a complete bipartite graph, that is the formulation used by the EMD algorithm.

The Network Simplex as implemented in the Lemon Graph library is clearly the fastest exact algorithm, and is able to solve every instance within a small fraction of a second. For this reason, in all our other numerical tests, we use that implementation.

\begin{figure}[ht]
	\begin{center}
	\includegraphics[width=\textwidth]{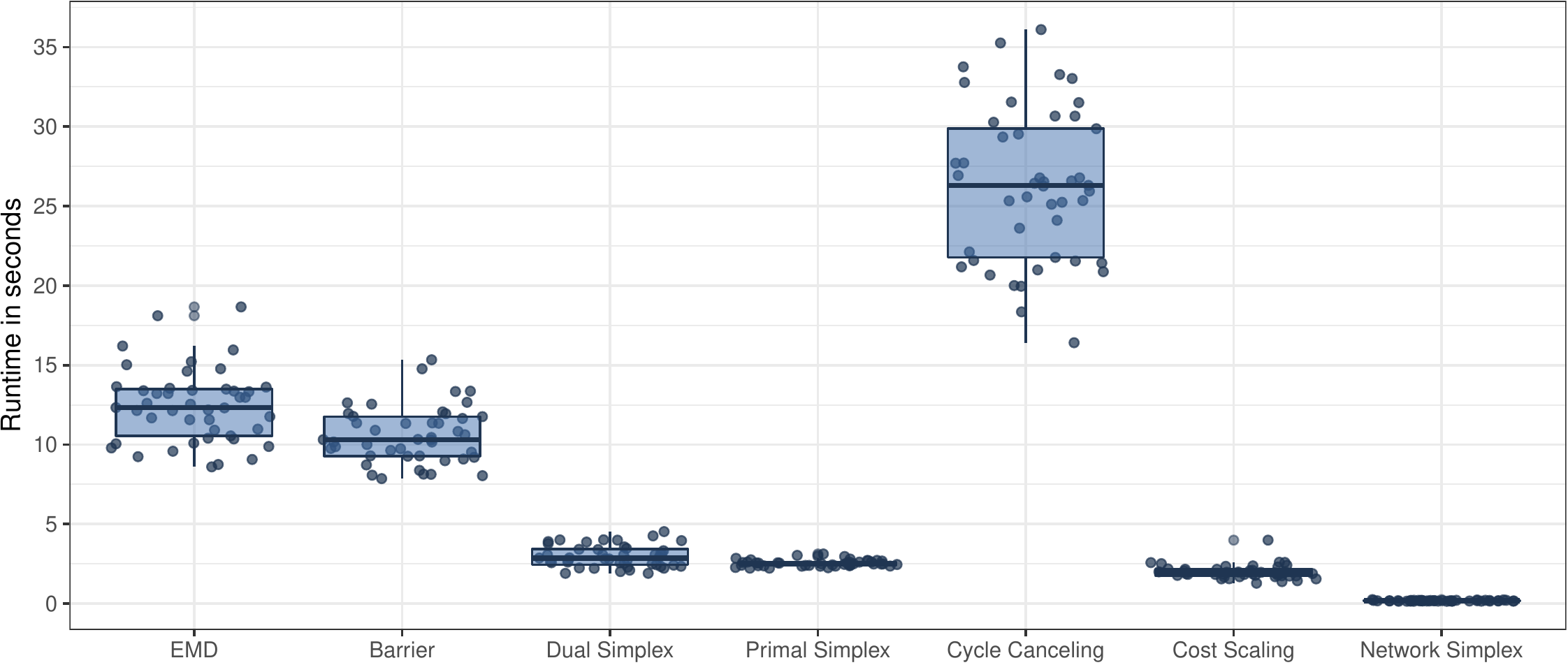}
	\caption{Comparison of different algorithms for solving the Minimum Cost Flow problem. Each single (jittered) dot gives the running time of computing the distance between a pair of images with the corresponding algorithm.}
	\label{Tab1}
	\end{center}
\end{figure}

\subsection{Exact distance computations}
As second step, we have compared the running times of the flow models proposed in this paper computing the exact Wasserstein distance of order 1 using the $d_1$, $d_\infty$ and $d_2$ ground distances.

First, we compare the running times for solving the models proposed in \cref{section:4}, with the running times of solving the equivalent problems formulated on bipartite graphs. Both type of problems are solved with the very efficient implementation of the Network Simplex algorithm provided by the Lemon Graph library.
Figure \ref{Tab3} shows the boxplots that summarize our results on the classical images of size $64 \times 64$. While the bipartite graphs are all of the same size, the problems defined with the $d_1$ distance seem to be simpler, and those with the $d_2$ distance look slightly harder. By using our reduced network flow models, we get a speedup of two orders of magnitude for $d_1$ and $d_\infty$: this is due to the significant reduction of the number of arcs. We recall that the bipartite graphs have $O(n^2)$ arcs, where $n=64 \times 64$, while our equivalent models have only $O(n)$ arcs. However, for the $d_2$ distance, the speedup is still present but more modest. In this case, our models have $O(\frac{6}{\pi^2}n^2)$ arcs, against the $O(n^2)$ of the bipartite graphs. In order to get obtain a more significant speedup while controlling the error for the $d_2$ distance, we have numerically evaluated our approximation scheme, as discussed in the next subsection.

\begin{figure}[ht]
	\begin{center}
	\includegraphics[width=\textwidth]{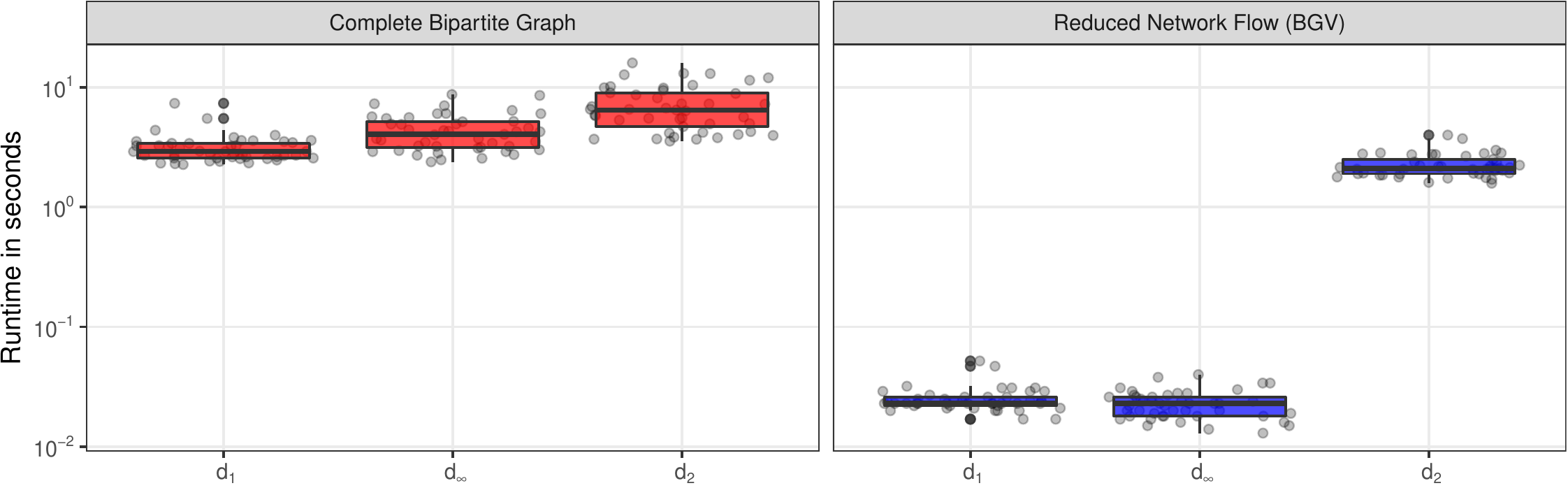}
	\caption{Running times of solving the same transportation instances on bipartite graphs versus using our reduced flow networks for classical images of size $64 \times 64$.}
	\label{Tab3}
	\end{center}
\end{figure}

Once we have completely abandoned the bipartite model, we have run a large set of tests using images of size up to $512\times 512$ with the three different types of ground distances.
Table \ref{tab:2} shows aggregate data elaborated from our results. For each combination of image size and ground distance, the table gives
the size of the flow network in terms of number of nodes $|V|$ and number of arcs $|A|$. Regarding the running times (in seconds), the same table
reports the average, the standard deviation, and the maximum runtime over 450 tests (45 for each of the 10 class of images).

From Table \ref{tab:2}, it is clear that the true challenge is to compute the Wasserstein distances using the $d_2$ ground distance,
since with the $d_1$ and $d_\infty$ ground distances we can solve all the images in at most 1000 seconds (though, in average, it takes 2 or 3 minutes),
while with the $d_2$ ground distance we cannot even solve the instances of size $256\times256$, since the code runs out of memory.

\begin{table}[ht!]
\centering
\begin{tabular}{cr@{\hspace{1em}}r@{\hspace{1em}}c@{\hspace{1em}}r@{\hspace{1em}}r@{\hspace{1em}}r}
Image size & \multicolumn{2}{c}{Graph size} & Ground & \multicolumn{3}{c}{Running times (sec)} \\
$N\times N$ & $|V|$ & $|A|$ & distance & Average & StdDev & Maximum \\
\hline\hline
 $32\times32$
&	1\,024	&	3\,968	&	$d_1$	&	0.002	&	0.001	&	0.009	\\
&		&	7\,812	&	$d_\infty$	&	0.002	&	0.001	&	0.009	\\
&		&	638\,692	&	$d_2$	&	0.075	&	0.016	&	0.154	\\
\hline\noalign{\smallskip}
$64\times64$
&	4\,096	&	16\,128	&	$d_1$	&	0.023	&	0.008	&	0.059	\\
&		&	32\,004	&	$d_\infty$	&	0.024	&	0.006	&	0.053	\\
&		&	10\,205\,236	&	$d_2$	&	2.011	&	0.573	&	4.401	\\
\hline\noalign{\smallskip}
$128\times128$
&	16\,384	&	65\,024	&	$d_1$	&	0.35	&	0.13	&	0.82	\\
&		&	129\,540	&	$d_\infty$	&		0.32	&	0.08	&   0.63	\\
&		&	163\,207\,372	&	$d_2$	&57.73	&	20.17	&	146.33	\\
\hline\noalign{\smallskip}
$256\times256$
&	65\,536	&	261\,120	&	$d_1$	&	6.7	&	3.0	&	23.7	\\
&		&	521\,220	&	$d_\infty$	&	5.5	&	1.7	&	12.9	\\
& \multicolumn{2}{r}{{\it out of memory}} & $d_2$ & - & - & -\\
\hline\noalign{\smallskip}
$512\times512$
&	262\,144	&	1\,046\,528	&	$d_1$	&	182.4	&	100.5	&	1056.5	\\
&		&	2\,091\,012	&	$d_\infty$	&	139.6	&	58.3	&	405.0	\\
& \multicolumn{2}{r}{{\it out of memory}} & $d_2$ & - & - & -\\
 \hline\noalign{\smallskip}
 \end{tabular}
\caption{Running times (in seconds) to compute exact Wasserstein distances. Each row gives the average, standard deviation, and maximum time over 450 instances.\label{tab:2}}
\end{table}

\subsection{Approximation error \texorpdfstring{$\CW_{d_2}$}{W2}: Theory vs. Practice}
In this section, we present the numerical results on the error obtained in practice when we compute the Wasserstein distance of order 1 with our approximation scheme.
The results are presented for the images of size $32 \times 32$, $64 \times 64$ and $128\times128$, for which we can compute both
the exact distance using $F_{G_{N-1}, d_2}$ and the approximate distance using $F_{G_L, d_2}$, with $L$ ranging in the set $\{2,3,5,10\}$.

We measure the relative percentage error, denoted by $\CE_{G_{L}}$, as the following ratio:

$$\CE_{G_{L}} := \frac{F_{G_L, d_2} - F_{G_{N-1}, d_2}}{F_{G_L, d_2}},$$

\noindent where $F_{G_{N-1}, d_2}$ is equal to the optimal Wasserstein distance, and $F_{G_L, d_2}$ is the value of our approximation
obtained by solving the uncapacitated min cost flow problem on the network built according to the parameter $L$.
Our goal is to compare the values of $\CE_{G_{L}}$ with the upper bound $\bar\Gamma_{G_{L}}$ predicted by \cref{PropL2bound}, that is

\[
	\bar\Gamma_{G_{L}} := 1 - \sqrt{\frac12+ \frac{L}{2\sqrt{1+L^2}}} \geq \CE_{G_{L}}.
\]
\cref{tab:3} reports the results for the images of size $32\times32$, $64\times64$ and $128 \times 128$. The columns specify, in order, the image size, the value of
the parameter $L$, the cardinality of the arc set $E_L$, the average runtime in seconds,
the upper bound $\bar\Gamma_{G_{L}}$ guaranteed by \cref{PropL2bound}, the empirical average of the errors $\CE_{G_{L}}$
obtained in practice, and the maximum of such errors.
Note that while from \cref{PropL2bound} in order to get an error smaller than $1\%$ we should use at least $L=5$, in practice, we get
on average such small errors already by using $L=3$. Indeed, we can set the value of parameter $L$ in such a way to achieve
the desired trade off between numerical precision in computing the Wasserstein distances and the running time we accept to wait.
\cref{tab:4} details the maximum error $\CE_{G_{L}}$ obtained for each class of images.
This shows that in practice a better trade off between the numerical precision of the distance and the running time
can be obtained by considering the type of 2D histograms of interest.

\begin{table}[ht!]
\centering
 \begin{tabular}{c@{\hspace{0.5em}}c@{\hspace{0.5em}}r@{\hspace{1em}}r@{\hspace{1em}}c@{\hspace{1em}}c@{\hspace{1em}}c}
Size& Parameter & $|E_L|$ & Runtime & $\bar\Gamma_{G_{L}}$ & mean$\{\CE_{G_{L}}\}$ & $\max \{ \CE_{G_{L}} \}$ \\
\hline\noalign{\smallskip}
$128\times128$
& $L=2$	 & 257\,556 & 0.446 & 2.675\% & 1.291\% & 2.646\%  \\
& $L=3$	 & 510\,556 & 0.748 & 1.291\% & 0.463\% & 1.153\%  \\
& $L=5$	 & 1\,254\,508 & 1.153 & 0.486\% & 0.121\% & 0.447\%  \\
& $L=10$ & 16\,117\,244 & 2.424 & 0.124\% & 0.019\% & 0.109\%  \\
\hline\noalign{\smallskip}
(exact) & $L=127$ &163\,207\,372	& 57.729 & & \\
\hline
\hline
\end{tabular}
\caption{Bound errors: Theory vs. Practice. Each row gives the averages over 450 instances.\label{tab:3}}
\end{table}
\begin{table}[ht!]
\centering
 \begin{tabular}{l@{\hspace{3em}}c@{\hspace{3em}}c@{\hspace{3em}}c@{\hspace{3em}}c}
 & \multicolumn{4}{c}{Maximum Bound percentage error $\CE_{G_{L}}$} \\
Class &	L=2	& L=3 & L=5 & L=10 \\
\hline\noalign{\smallskip}
CauchyDensity	    &	2.576\%	&	1.147\%	&	0.447\%	&	0.109\%	\\
ClassicImages	    &	1.813\%	&	0.716\%	&	0.246\%	&	0.053\%	\\
GRFmoderate	        &	2.365\%	&	0.937\%	&	0.348\%	&	0.065\%	\\
GRFrough	        &	1.918\%	&	0.780\%	&	0.279\%	&	0.039\%	\\
GRFsmooth	        &	2.248\%	&	1.069\%	&	0.338\%	&	0.068\%	\\
LogGRF	            &	2.546\%	&	1.019\%	&	0.370\%	&	0.076\%	\\
LogitGRF	        &	2.332\%	&	1.104\%	&	0.258\%	&	0.067\%	\\
MicroscopyImages	&	2.069\%	&	0.900\%	&	0.272\%	&	0.045\%	\\
Shapes	            &	2.670\%	&	1.168\%	&	0.317\%	&	0.094\%	\\
WhiteNoise	        &	0.908\%	&	0.243\%	&	0.043\%	&	0.005\%	\\
\hline\noalign{\smallskip}
$\bar\Gamma_{G_{L}}$ & 2.675\%	&	1.168\%	&	0.447\%	&	0.109\% \\
\end{tabular}
\caption{Maximum bound errors for each class of images in the benchmark. Each row gives the maximum over 135 instances, for size $N=32\times32, 64\times 64, 128\times 128$.\label{tab:4}}
\end{table}

Motivated by the results shown in \cref{tab:3}, we measured how our approximation scheme scales
for increasing image sizes, using $L\in\{2,3,5,10\}$ and by restricting our test to the Cauchy images, since
they are those with the larger error for $L=5$ and $L=10$.
For each combination of image size and value of $L$, \cref{tab:5} reports the graph size and the basic statistics
on the running time along the same line of the previous tables: average running time along with the respective standard deviations, and maximum running times.

\begin{table}[ht!]
\centering
\begin{tabular}{cc@{\hspace{1em}}r@{\hspace{1em}}r@{\hspace{1em}}r@{\hspace{1em}}r@{\hspace{1em}}r}
Image size & Param & \multicolumn{2}{c}{Graph size} & \multicolumn{3}{c}{Running times (sec)} \\
$N\times N$ & $L$ & $|V|$ & $|E_L|$ &  Average & StdDev & Maximum \\
\hline\noalign{\smallskip}
$256\times256$
& 2  & 65\,536 & 1\,039\,380  & 8.6 & 2.5 & 14.8 \\
& 3  &         & 2\,069\,596  & 8.7 & 2.7 & 20.9 \\
& 5  &         & 5\,129\,836  & 12.9 & 1.3 & 16.8 \\
& 10 &         & 16\,117\,244 & 24.6 & 4.7 & 35.3 \\
\hline\noalign{\smallskip}
$512\times512$
& 2  & 262\,144 & 4\,175\,892  & 170.8 &	97.2 &	487.5 \\
& 3  &          & 8\,333\,404  & 148.0 &	70.6 &	361.6 \\
& 5  &          & 20\,744\,812 & 171.9 &	29.7 &	236.3 \\
& 10 &          & 65\,782\,268 & 471.5 &	139.4 &	950.6 \\
\hline\noalign{\smallskip}
 \end{tabular}
\caption{Attacking larger Cauchy images with $d_2$ ground distance and parameter $L\in\{2,3,5,10\}$. Each row gives the averages over 45 instances.\label{tab:5}}
\end{table}

\subsection{Comparison with other approaches}
We compare our solution method with other two approaches: the first is the algorithm proposed in \cite{LingOkada2007} for the special case using the Manhattan as ground distance. The second approach is the (heuristic, with no guarantees) improved Sinkhorn's algorithm \cite{Cuturi2013} as implemented in \cite{Schmitzer2016}.

\paragraph{Manhattan ground distance \cite{LingOkada2007}} Figure \ref{Tab2} shows the comparison of running time in seconds of our method with the algorithm proposed in \cite{LingOkada2007}. The results refer to the running time for computing the distances between the 45 possible pairs of classical images, at resolution ranging from $32\times32$ up to $512\times512$. Since both methods are exact, the distance value is always the same (indeed, it is the optimum). Note that apart from the very small case of size $32\times32$, our approach is always faster, with the difference increasing for the larger dimensions. While the two methods are using the same network flow model, they are using different solution algorithms: from these results, we can conclude that the Network Simplex implemented in \cite{Kovacs2015} is faster than the implementation of combinatorial algorithm provided by the authors of \cite{LingOkada2007}, despite the fact that it has a lower worst-case time complexity.

\paragraph{Sinkhorn's algorithm \cite{Schmitzer2016}} Figure \ref{Tab3} shows the result of comparing the stabilized Sinkhorn's algorithm described in \cite{Schmitzer2016}, with our approximation scheme for the Euclidean ground distance described in \cref{section:4}. We remark that the Sinkhorn's algorithm, differently from our approach, is a heuristic method and it does not have any guarantee on the optimality gap of the solution it computes.
The results refer to the classical images of size $32\times32$, $64\times64$, and $128\times 128$.
The exact solutions are computed using our approximation scheme with $L=N-1$. For our approximation approach, we used the values $L\in \{2,3,5,10\}$. As expected, the running time grows with increasing values of $L$ and with increasing size of the images (note the log scale on the vertical axis). Our approach, is not only always faster, but, given the same value of empirical optimality gap, it can be two orders of magnitude faster. In addition, we remark that the Network Simplex runs in a single thread, while the implementation of the Sinkhorn's algorithm we use runs in parallel on 20 threads.

\begin{figure}[ht]
	\begin{center}
	\includegraphics[width=\textwidth]{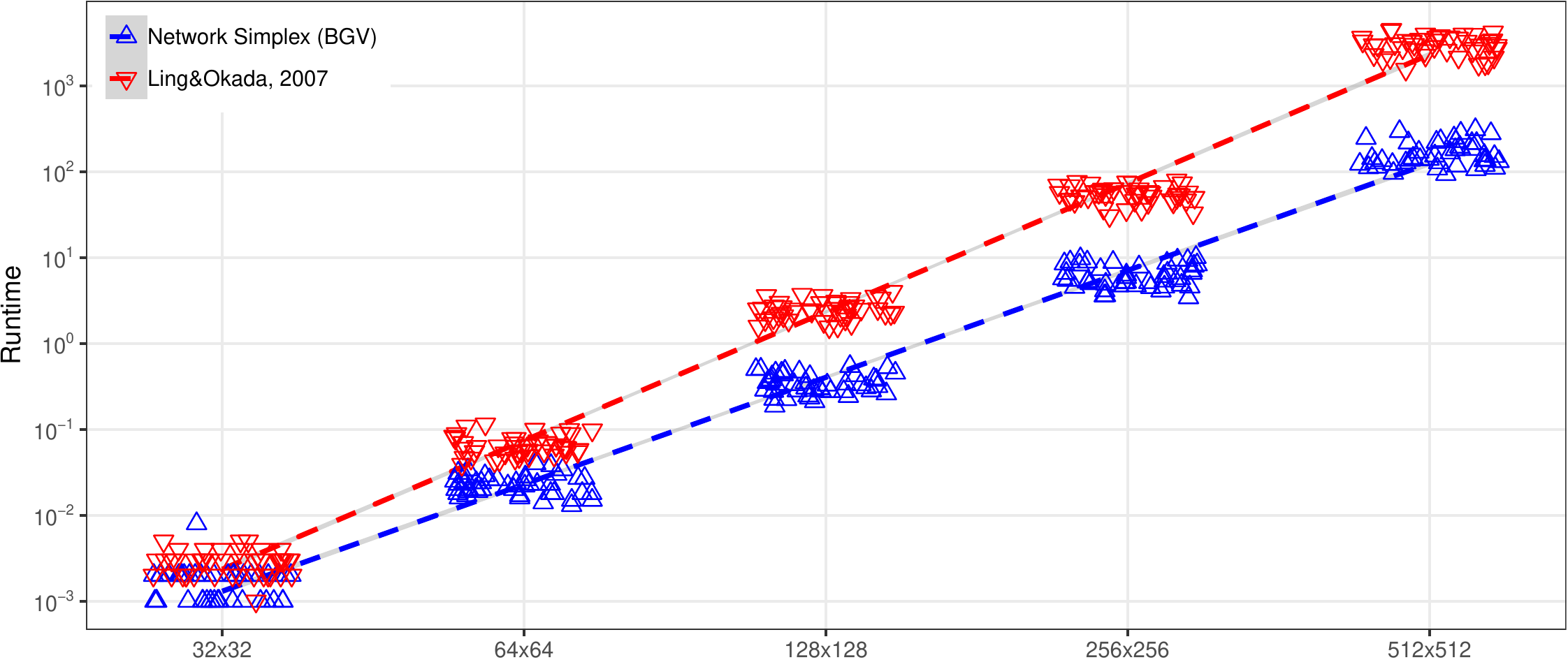}
	\caption{Comparison of runtime between the algorithm proposed in \cite{LingOkada2007} and our approached based on the Network Simplex algorithm using the Classical images.\label{Tab2}}
	\end{center}
\end{figure}

\begin{figure}[ht]
	\begin{center}
	\includegraphics[width=\textwidth]{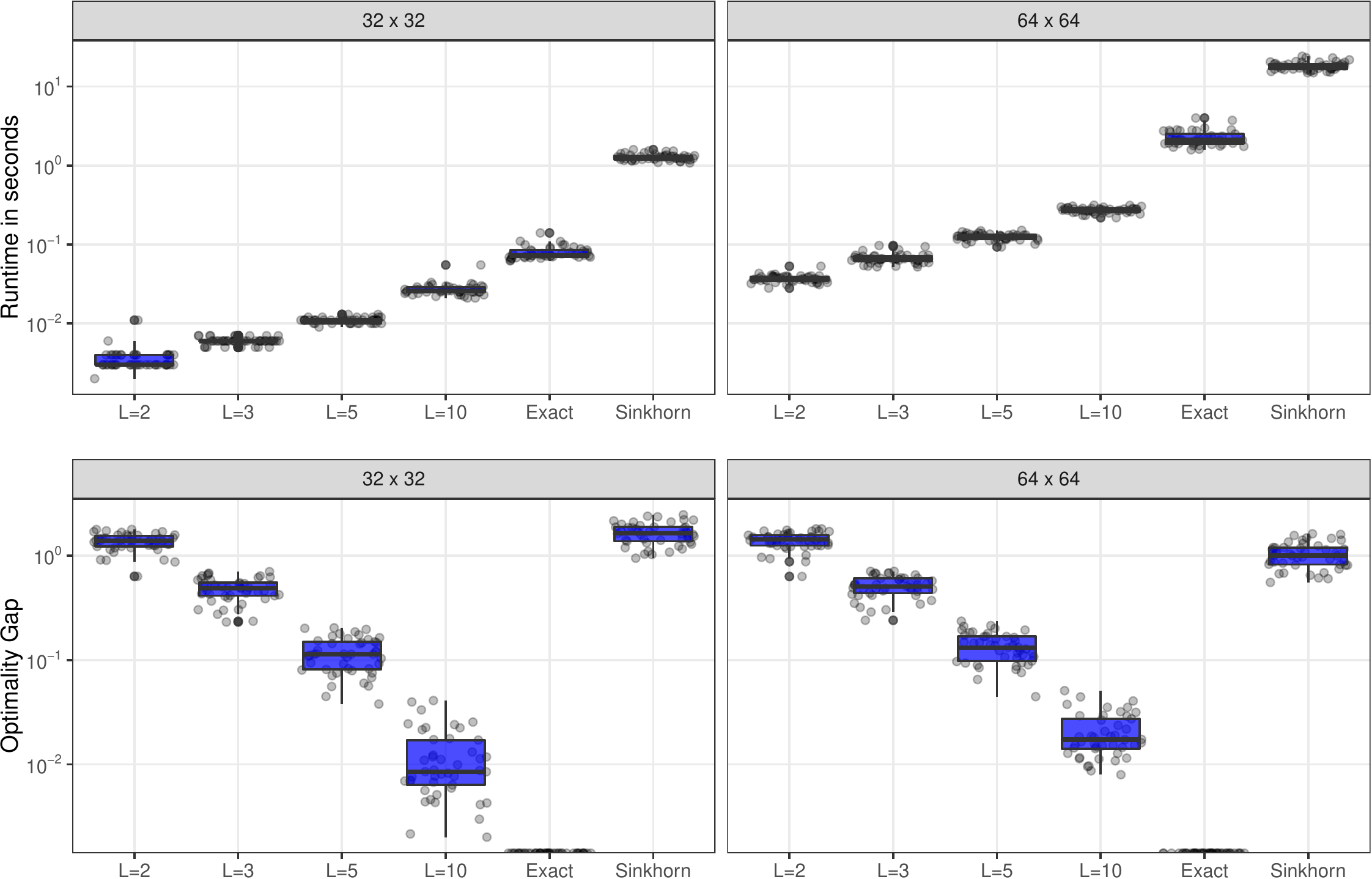}
	\caption{Comparison of runtime and optimality gap between the Sinkhorn's algorithm \cite{Cuturi2013} as implemented in \cite{Schmitzer2016}, and our approximation scheme discussed in \cref{section:4}, using the Classical images of sixe 32x32 and 64x64, with $L\in \{2,3,5,10\}$ (approximate) and $L=N-1$ (exact).\label{Tab5}}
	\end{center}
\end{figure}


\section{Conclusion}\label{section:6}

Comparing two histograms, or establishing which histograms, in a given set, are more alike, is the crucial question in a number of applications with industrial and scientific scope. The staggering availability of data coming, e.g., from the Internet or from biological sampling and imaging, is undoubtedly encouraging the digitalization of data measures (e.g., images), while, at the same time, requiring a computational effort that, as of today, may not be within reach of modern workstations.

Motivated by these considerations, in this article we have addressed the problem of rapidly computing the Kantorovich distance (also known as Wasserstein distance of order 1) between 2D histograms. In our approach, we translate the original discrete optimal transportation problem \cref{eq:kantorovich} to an uncapacitated minimum cost flow problem on a directed graph \cref{bflow}. In \cref{section:3} we prove that the two problems are equivalent when the Kantorovich cost $c$ is a distance.

The key observation is that, by reducing the size of the graph in the minimum cost flow, one can approximate the Kantorovich solution at a lower computational cost. Precisely, when the cost $c$ is the {\it taxicab} or the {\it maximum} distance, we are able to compute the optimal solution on a reduced flow network of size $O(n)$, where $N$ is  number of pixels in the images, i.e., bins of a histogram. When the cost $c$ is the Euclidean distance, the size of the network that yields the exact optimal solution is $O(\frac{6}{\pi^2}n^2)$. Our main contribution, in the Euclidean distance case, is that for any given error $\varepsilon >0$ we can provide a reduced network of size $O(n)$ that yields an approximate solution, which is at most $\varepsilon$ away from the exact one.
With our approximation method we were able to compute the distance between $512 \times 512$ images, with an error lower than 0.12\%, in less than 20 minutes.



\section*{Acknowledgments}
This research was partially supported by the Italian Ministry of Education, University and Research (MIUR):
Dipartimenti di Eccellenza Program (2018--2022) - Dept. of Mathematics ``F. Casorati'', University of Pavia.

We are deeply indebted to Giuseppe Savar\'e, for introducing us to optimal transportation and for many stimulating discussions and suggestions. We thanks Rico Zenklusen for a useful discussion concerning the proof of Proposition 6.
We thank Bernhard Schmitzer for suggesting the best parameters for his code.

\bibliographystyle{siamplain}
\bibliography{references}
\end{document}